\documentclass[oneside,11pt,reqno]{amsart}

\usepackage[a4paper,margin=30mm]{geometry}

\usepackage{verbatim,upref,amsxtra,amssymb,amscd,graphicx}

\usepackage{mathtools}

\usepackage{amsrefs}

\usepackage{epsfig,color}

\usepackage{varioref}

\usepackage{times}

\DeclareMathAlphabet\mathscr{U}{eus}{m}{n}
\SetMathAlphabet\mathscr{bold}{U}{eus}{b}{n}
\DeclareMathAlphabet\matheur{U}{eur}{m}{n}
\SetMathAlphabet\matheur{bold}{U}{eur}{b}{n}

\numberwithin{equation}{section}

\newtheorem{theorem}{Theorem}[section]
\newtheorem{proposition}[theorem]{Proposition}
\newtheorem{lemma}[theorem]{Lemma}
\newtheorem{corollary}[theorem]{Corollary}

\theoremstyle{definition}

\newtheorem{definition}[theorem]{Definition}
\newtheorem{example}[theorem]{Example}

\theoremstyle{remark}

\newtheorem{remarks}[theorem]{Remarks}

\newcommand{\htop}{\textup{h}_\textup{top}}

\newcommand{\hash}{{\scriptscriptstyle\#}}

\date{}

\begin{document}
\allowdisplaybreaks\frenchspacing

\title[Generators and Symbolic Representations of Algebraic Group Actions]{Intrinsic Ergodicity, Generators and Symbolic Representations\linebreak of Algebraic Group Actions}

\author{Hanfeng Li}
\thanks{The first author gratefully acknowledges partial support by the NSF grant DMS-1900746}

\address{Hanfeng Li: Department of Mathematics, SUNY at Buffalo, NY 14260-2900, USA}
\email{hfli@math.buffalo.edu}

\author{Klaus Schmidt}

\address{Klaus Schmidt: Mathematics Institute, University of Vienna, Oskar-Morgenstern-Platz 1, A-1090 Vienna, Austria}
\email{klaus.schmidt@univie.ac.at}


\dedicatory{Dedicated to Anatole M. Vershik on the occasion of his 90th birthday}

	\begin{abstract}
We construct natural symbolic representations of intrinsically ergodic, but not necessarily expansive, principal algebraic actions of countably infinite amenable groups and use these representations to find explicit generating partitions (up to null-sets) for such actions.
	\end{abstract}

\maketitle

\section{Introduction}
	\label{s:introduction}

Let $\Gamma $ be a countably infinite discrete group with integral and real group rings $\mathbb{Z}\Gamma \subset \mathbb{R}\Gamma $. Every $g\in\mathbb{R}\Gamma $ is written as a finite formal sum $g =\sum_{\gamma \in \Gamma }g_{\gamma }\cdot \gamma $ with $g_{\gamma }$ in $\mathbb{Z}$ or $\mathbb{R}$, respectively, for every $\gamma $. We write $\textup{supp}(g)=\{\gamma \in \Gamma \mid g_\gamma \ne 0\}$ for the \emph{support} of $g$ and set $g^+= \sum_{\gamma \in \Gamma }\max\{g_\gamma ,0\}\cdot \gamma $ and $g^-= \sum_{\gamma \in \Gamma }\min\{g_\gamma ,0\}\cdot \gamma $.

\smallskip For $g=\sum_{\gamma \in\Gamma }g_\gamma \cdot \gamma ,\,h=\sum_{\gamma \in\Gamma }h_\gamma \cdot \gamma $ in $\mathbb{R}\Gamma $ we denote by $g+h=\sum_{\gamma \in \Gamma }(g_\gamma +h_\gamma )\cdot \gamma $ their \emph{sum}, by $gh=\sum_{\gamma ,\delta \in\Gamma }g_\gamma h_\delta \cdot \gamma \delta $ their \emph{product}, and by $g^*=\sum_{\gamma \in\Gamma }g_\gamma \cdot \gamma ^{-1}$ and $h^*=\sum_{\gamma \in\Gamma }h_\gamma \cdot \gamma ^{-1}$ their \emph{adjoints}. The adjoint map $g\mapsto g^*$ is an \emph{involution} on $\mathbb{R}\Gamma $, i.e., $(gh)^*=h^*g^*$.

\smallskip An \emph{algebraic $\Gamma $-action} is a homomorphism $\tau \colon \Gamma \rightarrow \textup{Aut}(X)$ from $\Gamma $ to the group of continuous automorphisms of a compact metrizable abelian group $X$. If $\tau $ is such an algebraic $\Gamma $-action, then $\tau ^\gamma \in \textup{Aut}(X)$ denotes the image of $\gamma \in \Gamma $, and $\tau ^{\gamma \delta }=\tau ^\gamma \tau ^{\delta }$ for every $\gamma ,\delta \in \Gamma $. The action $\tau $ induces an action of $\mathbb{Z}\Gamma $ by group homomorphisms $\tau ^f\colon X\rightarrow X$, where $\tau ^f=\sum_{\gamma \in\Gamma }f_\gamma \tau ^\gamma $ for every $f=\sum_{\gamma \in\Gamma }f_\gamma \cdot \gamma \in\mathbb{Z}\Gamma $. Clearly, if $f,g\in\mathbb{Z}\Gamma $, then $\tau ^{fg}=\tau ^f \tau ^g$.

\smallskip Let $\hat{X}$ be the dual group of $X$. If $\hat{\tau }^\gamma $ is the automorphism of $\hat{X}$ dual to $\tau ^\gamma $, then the map $\hat{\tau }\colon \Gamma \rightarrow \textup{Aut}(\hat{X})$ satisfies that $\hat{\tau }^{\gamma \delta }=\hat{\tau }^{\delta }\hat{\tau }^\gamma $ for all $\gamma ,\delta \in \Gamma $. We denote by $\hat{\tau }^f\colon \hat{X}\rightarrow \hat{X}$ the group homomorphism dual to $\tau ^f$ and set $f\cdot a=\hat{\tau }^{f^*}a$ for every $f\in\mathbb{Z}\Gamma $ and $a\in \hat{X}$. The resulting map $(f,a)\mapsto f\cdot a$ from $\mathbb{Z}\Gamma \times \hat{X}$ to $\hat{X}$ satisfies that $(fg)\cdot a=f\cdot (g\cdot a)$ for all $f,g\in\mathbb{Z}\Gamma $ and turns $\hat{X}$ into a left module over the group ring $\mathbb{Z}\Gamma $. Conversely, if $M$ is a countable left module over $\mathbb{Z}\Gamma $, we set $X=\widehat{M}$ and put $\hat{\tau }^fa=f^*\cdot a$ for $f\in\mathbb{Z}\Gamma $ and $a\in M$. The maps $\tau ^f\colon \widehat{M}\rightarrow \widehat{M}$ dual to $\hat{\tau }^f$ define an action of $\mathbb{Z}\Gamma $ by homomorphisms of $\widehat{M}$, which in turn induces an algebraic action $\tau $ of $\Gamma $ on $X=\widehat{M}$.

The simplest examples of algebraic $\Gamma $-actions arise from $\mathbb{Z}\Gamma $-modules of the form $M=\mathbb{Z}\Gamma /(f)$, where $(f) = \mathbb{Z}\Gamma f$ is the principal left ideal generated by $f$: such actions are called \emph{principal}. For an explicit description of such actions we put $\mathbb{T}=\mathbb{R}/\mathbb{Z}$ and define the left and right shift-actions $\lambda $ and $\rho $ of $\Gamma $ on $\mathbb{T}^\Gamma $ by
	\begin{equation}
	\label{eq:lambda}
(\lambda ^\gamma x)_{\delta }=x_{\gamma ^{-1}\delta }\quad \textup{and}\quad (\rho ^\gamma x)_{\delta }=x_{\delta \gamma }
	\end{equation}
for every $\gamma \in \Gamma $ and $x=(x_{\delta })_{\delta \in \Gamma }\in\mathbb{T}^\Gamma $. The actions $\lambda $ and $\rho $ extend to $\mathbb{Z}\Gamma $-actions on $\mathbb{T}^\Gamma $ given by
	\begin{displaymath}
\lambda ^f=\textstyle\sum_{\gamma \in\Gamma }f_\gamma \lambda ^\gamma ,\qquad \rho ^f=\textstyle\sum_{\gamma \in\Gamma }f_\gamma \rho ^\gamma ,
	\end{displaymath}
for every $f= \sum_{\gamma \in\Gamma }f_\gamma \cdot \gamma \in\mathbb{Z}\Gamma $. These $\mathbb{Z}\Gamma $-actions obviously commute: for every $f,g\in \mathbb{Z}\Gamma $ and $x\in \mathbb{T}^\Gamma $,
	\begin{displaymath}
(\lambda ^f\circ \rho ^g )\,x = (\rho ^g\circ \lambda ^f)\,x .
	\end{displaymath}

\smallskip The pairing $\langle f,x\rangle =\sum_{\gamma \in \Gamma }f_\gamma x_\gamma = (\rho ^fx)_{1_\Gamma }$ with $f=\sum_{\gamma \in \Gamma }f_\gamma \cdot \gamma \in\mathbb{Z}\Gamma $ and $x=(x_\gamma )\in\mathbb{T}^\Gamma $, identifies $\mathbb{Z}\Gamma $ with the dual group $\widehat{\mathbb{T}^\Gamma }$ of $\mathbb{T}^\Gamma $ and has the property that
	\begin{align*}
\langle h,\rho ^fx\rangle &=\smash[t]{\Bigl\langle h,\sum\nolimits_{\delta \in \Gamma }f_{\delta }\rho ^{\delta }x\Bigr\rangle =\sum\nolimits_{\gamma \in \Gamma }h_\gamma \sum\nolimits_{\delta \in \Gamma }f_{\delta }x_{\gamma \delta }}
	\\
&=\sum\nolimits_{\gamma \in \Gamma }\sum\nolimits_{\delta \in \Gamma }h_{\gamma \delta ^{-1}}f_{\delta }x_\gamma =\sum\nolimits_{\gamma \in \Gamma }(hf)_\gamma x_\gamma =\langle hf,x\rangle
	\end{align*}
for every $f,h\in \mathbb{Z}\Gamma $ and $x\in \mathbb{T}^\Gamma $. Every $f\in \mathbb{Z}\Gamma $ defines a $\lambda $-invariant closed subgroup
	\begin{equation}
	\label{eq:Xf}
	\begin{aligned}
X_f &=\ker \rho ^f = \{x\in \mathbb{T}^\Gamma \mid \rho ^fx=0\}
	\\
&= \{x\in \mathbb{T}^\Gamma \mid \langle h,\rho ^fx \rangle =\langle hf,x\rangle = 0\enspace \textup{for every}\enspace h\in \mathbb{Z}\Gamma \}= (f)^\perp \subset \widehat{\mathbb{Z}\Gamma } =\mathbb{T}^\Gamma .
	\end{aligned}
	\end{equation}
We denote by $\lambda _f = \lambda _{X_f}$ the restriction of $\lambda $ to $X_f$ and note that the normalized Haar measure $\mu _f$ on $X_f$ is invariant under $\lambda _f$.

	\begin{definition}
	\label{d:principal}
For every $f\in \mathbb{Z}\Gamma $ we call the left shift-action $\lambda _f$ on the probability space $(X_f,\mu _f)$ the \emph{principal algebraic $\Gamma $-action} defined by $f$.
	\end{definition}

Dynamical properties of algebraic actions of countably infinite groups --- and, in particular, of principal actions --- have been investigated at various levels of generality (cf. e.g., \cite{DSAO}, \cite{Chung-Li}, \cite{Li-Thom}, or \cite{Hayes}). In this paper we focus on symbolic representations of principal algebraic actions of countably infinite amenable groups and on generators of such actions arising from these representations.

\medskip Symbolic representations of algebraic actions have a long history. The first such representations arose from geometrically constructed Markov partitions around 1967--1970 (cf. \cite{Sinai}, \cite{Bowen}) and helped to provide a crucial link between smooth and symbolic dynamics. A different approach to symbolic representations of toral automorphisms had its origins in the paper \cite{Vershik} by Vershik from 1992, where he represented the hyperbolic toral automorphism \smash{$\bigl(
	\begin{smallmatrix}
1&1
	\\
1&0
	\end{smallmatrix}
\bigr)$} by the \emph{Golden Mean} shift by using homoclinic points rather than Markov partitions. Vershik's original construction was subsequently extended to arbitrary hyperbolic toral and solenoidal automorphisms in \cite{K-V} and \cite{L-B}, and to the ‘homoclinic’ construction of symbolic covers of expansive\footnote{\,A continuous $\Gamma $-action $\tau $ on a compact metrizable space $Y$ with compatible metric $\mathsf{d}$ is \emph{expansive} if there exists a $\delta >0$ such that $\sup_{\gamma \in \Gamma }\mathsf{d}(\tau ^\gamma y,\tau ^\gamma y')\ge \delta $ whenever $y,y'$ are distinct points in $Y$.
	\label{f:expansive}
} principal algebraic $\mathbb{Z}^d$-actions (cf. \cite{E-S}). What these constructions have in common is that they use a \emph{summable homoclinic point} $w\in X_f$ of a principal algebraic action $\lambda _f$ of a countably infinite discrete group $\Gamma $ to define a shift-equivariant surjective map $\xi _w\colon \ell ^\infty (\Gamma ,\mathbb{Z})\to X_f$, and to restrict this map to a suitable compact shift-invariant subset $\mathcal{V}\subset \ell ^\infty (\Gamma ,\mathbb{Z})$ (cf. e.g. \cite{E-S}, \cite{Goell1}, or \cite{Lind-S-2}).

While expansive principal algebraic actions always have summable homoclinic points permitting such a construction, this is generally not the case for nonexpansive actions (cf. \cite{Chung-Li}, \cite{Deninger-S}, \cite{homoclinic}). In this paper we obviate the need for summable homoclinic points and show directly that, for every countably infinite discrete amenable group $\Gamma $ and every $f\in \mathbb{Z}\Gamma $ for which the principal algebraic action $\lambda _f$ on $(X_f,\mu _f)$ is \emph{intrinsically ergodic} there exists a natural isomorphism $(\textup{mod}\;\mu _f)$ of $\lambda _f$ with the left shift action $\bar{\lambda }$ of $\Gamma $ on a closed, shift-invariant subset $\bar{Z}_f$ of the symbolic space $\{-\|f^-\|_1,\dots ,\|f^+\|_1\}^\Gamma $, furnished with a shift-invariant Borel probability measure $\nu _f^\hash$ (Theorem \ref{t:symbolic}). As an obvious consequence of this isomorphism one obtains that the `alphabet' $\mathcal{B}_f\subsetneq \{-\|f^-\|_1,\dots ,\|f^+\|_1\}$ of $\bar{Z}_f$ determines a natural generator\footnote{\,If $\tau $ is a measure-preserving action of a countably infinite group $\Gamma $ on a standard probability space $(Y,\mathcal{B}_Y,\mu )$, a countable Borel partition $\mathcal{C}$ of $Y$ is a \emph{generator} for $\tau $ if the smallest $\tau $-invariant sigma-algebra $\mathcal{T}\subset \mathcal{B}_Y$ containing $\mathcal{C}$ is equal to $\mathcal{B}_Y$ $(\textup{mod}\;\mu )$, i.e. up to $\mu $-null sets.} for $\lambda _f$ on $(X_f, \mu _f)$ (Corollary \ref{c:B}). As a further corollary of this construction we see that the partition $\mathcal{C}_f = \{C_j\mid j=0,\dots ,\|f\|_1-1\}$ of $X_f$, defined by
	\begin{displaymath}
C_j = \bigl\{x = (x_\gamma )_{\gamma \in \Gamma } \in X_f \mid {\scriptstyle\tfrac{j}{\|f\|_1}}\le x_{1_\Gamma } < {\scriptstyle\tfrac{j+1}{\|f\|_1}} \enspace (\textup{mod}\;1)\bigr\}
	\end{displaymath}
for $j=0,\dots ,\|f\|_1-1$, is a generator $(\textup{mod}\;\mu _f)$ for $\lambda _f$ (Corollary \ref{c:C}). If $\lambda _f$ is \emph{expansive}, $\mathcal{C}_f$ is obviously a generator without any conditions on $\Gamma $ (cf. Subsection \ref{sss:expansive}), but for nonexpansive actions this result is nontrivial.

In Section \ref{s:examples} we present examples of intrinsically ergodic principal algebraic actions $\lambda _f$, $f\in \mathbb{Z}\Gamma $, of countably infinite discrete amenable groups $\Gamma $. If $\Gamma =\mathbb{Z}^d$, $d\ge1$, or if $\Gamma $ is arbitrary and $\lambda _f$ is expansive, the situation is well-understood (Subsections \ref{ss:Z-d} or \ref{sss:expansive}). For nonexpansive principal algebraic actions intrinsic ergodicity is a more elusive property. A sufficient condition for intrinsic ergodicity is that the group $\Delta ^1(X_f)$ of summable homoclinic points of $\lambda _f$ is dense in the group $X_f$ carrying the action (Proposition \ref{p:homoclinic}). In Theorem \ref{T-Harmonic} we verify the latter condition for \emph{well-balanced} polynomials $f\in \mathbb{Z}\Gamma $, provided that $\Gamma $ is not virtually $\mathbb{Z}$ or $\mathbb{Z}^2$ and the center of $\Gamma $ contains an element of infinite order.

	\section{Linearization of principal algebraic actions}
	\label{s:linearization}

Let $\Gamma $ be a countably infinite discrete group, and let $\ell ^\infty (\Gamma ,\mathbb{R})$ be the space of all bounded maps $v\colon \Gamma \to \mathbb{R}$, furnished with the norm $\|v\|_\infty =\sup_{\gamma \in \Gamma }|v_\gamma |$. We write $\eta \colon \ell ^\infty (\Gamma ,\mathbb{R})\to \mathbb{T}^\Gamma $ for the weak$^*$-continuous map defined by
	\begin{equation}
	\label{eq:eta}
\eta(w)_\gamma = w_\gamma \;\text{(mod 1)},\,\gamma \in \Gamma ,
	\end{equation}
and define the shift-actions $\bar{\lambda }$ and $\bar{\rho }$ of $\Gamma $ on $\ell ^\infty (\Gamma ,\mathbb{R})$ as in \eqref{eq:lambda} by
	\begin{equation}
	\label{eq:bar-shifts}
(\bar{\lambda }^\gamma v)_{\delta }= v_{\gamma ^{-1}\delta },\enspace \enspace (\bar{\rho }^\gamma v)_{\delta }= v_{\delta \gamma },
	\end{equation}
for every $v=(v_\delta )_{\delta \in \Gamma }\in\ell ^\infty (\Gamma ,\mathbb{R})$ and $\gamma \in\Gamma $. Again we extend these $\Gamma $-actions to $\mathbb{Z}\Gamma $-actions on $\ell ^\infty (\Gamma ,\mathbb{R})$ by setting
	\begin{displaymath}
\bar{\lambda }^h=\sum\nolimits_{\gamma \in\Gamma }h_\gamma \bar{\lambda }^\gamma \quad \textup{and}\quad \bar{\rho }^h=\sum\nolimits_{\gamma \in\Gamma }h_\gamma \bar{\rho }^\gamma
	\end{displaymath}
for every $h=\sum_{\gamma \in \Gamma }h_\gamma \cdot \gamma \in \mathbb{Z}\Gamma $. These actions correspond to the usual convolutions
	\begin{equation}
	\label{eq:convolutions}
\bar{\lambda }^hv=h\cdot v,\qquad \bar{\rho }^hv=v\cdot h^*,
	\end{equation}
for $h\in \mathbb{Z}\Gamma $ and $v \in \mathbb{R}\Gamma $, and extend further to $h\in \ell ^1(\Gamma ,\mathbb{R})$ and $v\in \ell ^\infty (\Gamma ,\mathbb{R})$.

\smallskip We fix a nonzero element $f\in \mathbb{Z}\Gamma $ and consider the left shift action $\lambda _f$ on the compact group $X_f\subset \mathbb{T}^\Gamma $ defined in \eqref{eq:lambda} -- \eqref{eq:Xf}. The space
	\begin{equation}
	\label{eq:Wf}
	\begin{aligned}
W_f\coloneqq\eta^{-1}(X_f)&= \{w\in\ell ^\infty (\Gamma ,\mathbb{R})\mid \eta(w)\in X_f\}
	\\
&=\{w\in\ell ^\infty (\Gamma ,\mathbb{R})\mid \bar{\rho }^fw = w\cdot f^*\in\ell ^\infty (\Gamma ,\mathbb{Z})\}
	\end{aligned}
	\end{equation}
is the \emph{linearization} of $X_f$, and the restriction of $\bar{\lambda }$ to $W_f$ is the \emph{linearization} of $\lambda _f$. We set
	\begin{equation}
	\label{eq:Yf-Zf}
Y_f= W_f \cap [0,1)^\Gamma \subset \ell ^\infty (\Gamma ,\mathbb{R}), \qquad Z_f = \bar{\rho }^f(Y_f)\subset \ell ^\infty (\Gamma ,\mathbb{Z}),
	\end{equation}
write $\bar{Y}_f$ and $\bar{Z}_f$ for the weak$^*$ closures of $Y_f$ and $Z_f$ in $\ell ^\infty (\Gamma ,\mathbb{R})$. Since $0\le y_\gamma <1$ for every $y\in Y_f$ and $\gamma \in \Gamma $, it is clear that
	\begin{displaymath}
c_f^- \le (\bar{\rho }^fy)_\gamma \le c_f^+
	\end{displaymath}
for every $y\in Y_f$ and $\gamma \in \Gamma $, where
	\begin{displaymath}
c_f^- = \min\,\{0,1-\|f^-\|_1\}\quad \textup{and}\quad c_f^+=\max\,\{0,\|f^+\|_1-1\}.
	\end{displaymath}
It follows that
	\begin{equation}
	\label{eq:bounds}
Z_f\subseteq \bar{Z}_f \subseteq \{c_f^-,\dots ,c_f^+\}^\Gamma .
	\end{equation}

Finally we denote by
	\begin{equation}
	\label{eq:Kf}
K_f=\{w\in\ell ^\infty (\Gamma ,\mathbb{R})\mid \bar{\rho }^fw=0\}\subset W_f
	\end{equation}
the kernel of $\bar{\rho }^f$ in $\ell ^\infty (\Gamma ,\mathbb{R})$. By \cite{Deninger-S}*{Theorem 3.2}, the action $\lambda _f$ on $X_f$ is expansive if and only if $K_f = \{0\}$.

If there is any danger of confusion we denote the restrictions of $\bar{\lambda }$ to the $\bar{\lambda }$-invariant sets $\bar{Y}_f$, $\bar{Z}_f$, and $K_f$ by $\bar{\lambda }_{\bar{Y}_f}$, $\bar{\lambda }_{\bar{Z}_f}$ and $\bar{\lambda }_{K_f}$, respectively. The map $\eta \colon \ell ^\infty (\Gamma ,\mathbb{R})\to \mathbb{T}^\Gamma $ in \eqref{eq:eta} induces a left shift-equivariant, continuous, surjective map from $\bar{Y}_f$ to $X_f$ whose restriction to $Y_f$ is bijective, and $\bar{\rho }^f$ intertwines the $\Gamma $-actions $\bar{\lambda }_{\bar{Y}_f}$ and $\bar{\lambda }_{\bar{Z}_f}$.

	\begin{proposition}
	\label{p:conjugacy}
Let $\Gamma $ be a countably infinite discrete group, $0\ne f\in \mathbb{Z}\Gamma $, and let $W_f, \bar{Y}_f, \bar{Z}_f, K_f$ be the closed, $\bar{\lambda }$-invariant subsets of $\ell ^\infty (\Gamma ,\mathbb{R})$ defined in \eqref{eq:Wf} -- \eqref{eq:Kf}. Put $\tilde{Z}_f= \bar{Z}_f\times K_f$, and denote by $\tilde{\lambda } = \bar{\lambda }_{\bar{Z}_f}\times \bar{\lambda }_{K_f}$ the product $\Gamma $-action on $\tilde{Z}_f$. Then there exists, for every $\bar{\lambda }$-invariant Borel probability measure $\nu $ on $\bar{Y}_f$, a $\tilde{\lambda }$-invariant Borel probability measure $\tilde{\nu }$ on $\tilde{Z}_f$ with the following properties:
	\begin{enumerate}
	\item
$\pi ^{(1)}_*\tilde{\nu }= \bar{\rho }^f_*\nu \eqqcolon \nu ^\hash$, where $\pi ^{(1)}\colon \tilde{Z}_f\to \bar{Z}_f$ is the first coordinate projection;
	\item
$\tilde{\nu }(\bar{Z}_f\times B_2(K_f))=1$, where $B_r(K_f)=\{w\in K_f\mid \|w\|_\infty \le r\}$ for every $r\ge0$;
	\item
The $\Gamma $-actions $\bar{\lambda }$ on $(\bar{Y}_f,\nu )$ and $\tilde{\lambda }$ on $(\tilde{Z}_f,\tilde{\nu })$ are measurably conjugate.
	\end{enumerate}
	\end{proposition}

	\begin{proof}
Since $\bar{\rho }^f$ is continuous, \cite{Partha}*{Theorem I.4.2} shows that there exists a Borel map $\zeta \colon \bar{Z}_f\rightarrow \bar{Y}_f$ with $\bar{\rho }^f\circ \zeta (z)=z$ for every $z\in \bar{Z}_f$. For every $z\in \bar{Z}_f$ and $\gamma \in \Gamma $ we set
	\begin{equation}
	\label{eq:cocycle}
c(\gamma ,z)=\zeta \circ \bar{\lambda }^\gamma (z)-\bar{\lambda }^\gamma \circ \zeta (z)\in B_1(K_f).
	\end{equation}
Then
	\begin{displaymath}
c(\gamma \delta ,z)=c(\gamma ,\bar{\lambda }^\delta z)+\bar{\lambda }^\gamma c(\delta ,z)
	\end{displaymath}
for every $\gamma ,\delta \in \Gamma $ and $z\in \bar{Z}_f$, i.e., the Borel map $c\colon \Gamma \times \bar{Z}_f\rightarrow K_f$ is a cocycle taking values in $B_1(K_f)$. We define a Borel action $\tilde{\lambda }_1$ of $\Gamma $ on $\tilde{Z}_f$ by setting
	\begin{displaymath}
	\label{eq:actions}
\tilde{\lambda }_1^\gamma (z,v) = (\bar{\lambda }^\gamma z,\bar{\lambda }^\gamma v-c(\gamma ,z))
	\end{displaymath}
for every $(z,v)\in\tilde{Z}_f$, and consider the injective Borel map $\theta _1\colon \bar{Y}_f\rightarrow \bar{Z}_f\times B_1(K_f)\subset \tilde{Z}_f$ given by
	\begin{equation}
	\label{eq:theta-1}
\theta _1(w)=(\bar{\rho }^fw,w-\zeta \circ \bar{\rho }^f(w))
	\end{equation}
for every $w\in \bar{Y}_f$. Then
	\begin{equation}
	\label{eq:equi-11}
\theta _1\circ \bar{\lambda }^\gamma =\tilde{\lambda }_1^\gamma \circ \theta _1
	\end{equation}
for every $\gamma \in\Gamma $.

\smallskip Let $\nu $ be a $\bar{\lambda }$-invariant Borel probability measure on $\bar{Y}_f$ and let $\nu ^\hash=\bar{\rho }^f_*\nu $. The probability measure $\tilde{\nu }^{(1)} = (\theta _1)_*\nu $ is $\tilde{\lambda }_1$-invariant by \eqref{eq:equi-11}, and is supported in the weak$^*$-compact and metrizable set $\bar{Z}_f\times B_1(K_f)\subset \tilde{Z}_f$. Furthermore, $\pi ^{(1)}_*\tilde{\nu }^{(1)} = \nu ^\hash$, where $\pi ^{(1)}\colon \tilde{Z}_f \to \bar{Z}_f$ is the first coordinate projection. We decompose $\tilde{\nu }^{(1)}$ over $\bar{Z}_f$ by choosing a Borel measurable family $\tilde{\nu }^{(1)}_z,\,z\in \bar{Z}_f$, of Borel probability measures on $K_f$ with $\tilde{\nu }^{(1)}_z(B_1(K_f))=1$ for every $z\in \bar{Z}_f$, and with
	\begin{displaymath}
\int g(z,v)\,d\tilde{\nu }^{(1)}(z,v)=\int_{\bar{Z}_f}\int_{K_f}g(z,v)\, d\tilde{\nu }^{(1)}_z(v)\,d\nu ^\hash(z)
	\end{displaymath}
for every bounded Borel map $g\colon \tilde{Z}_f \rightarrow \mathbb{R}$. Since $\tilde{\nu }^{(1)}$ is $\tilde{\lambda }_1$-invariant,
	\begin{equation}
	\label{eq:transform}
\int h(v)\,d\tilde{\nu }^{(1)}_z(v)=\int h(\bar{\lambda }^{\gamma }v-c(\gamma ,\bar{\lambda }^{\gamma ^{-1}}z))\,d\tilde{\nu }^{(1)}_{\bar{\lambda }^{\gamma ^{-1}}z}(v)
	\end{equation}
for every bounded Borel map $h\colon K_f\rightarrow \mathbb{R}$, every $\gamma \in \Gamma $, and $\nu ^\hash\textit{-a.e.}\;z\in \bar{Z}_f$.

\smallskip Define a Borel map $b\colon \bar{Z}_f\to K_f$ by setting $b(z) = \int_{K_f} v\,d\tilde{\nu }^{(1)}_z(v) \in B_1(K_f)$ for every $z\in \bar{Z}_f$, where the integral is taken coordinate-wise (or, equivalently, in the weak$^*$-topology) on $\ell ^\infty (\Gamma ,\mathbb{R})$. Equation \eqref{eq:transform} shows that
	\begin{align*}
b(z)&= \int_{K_f} v\,d\tilde{\nu }^{(1)}_z(v) = \int_{K_f} (\bar{\lambda }^{\gamma }v - c(\gamma ,\bar{\lambda }^{\gamma ^{-1}}z)) \,d\tilde{\nu }^{(1)}_{\bar{\lambda }^{\gamma ^{-1}}z}(v)
	\\
&= \int_{K_f} \bar{\lambda }^{\gamma }v \,d\tilde{\nu }^{(1)}_{\bar{\lambda }^{\gamma ^{-1}}z}(v) - c(\gamma ,\bar{\lambda }^{\gamma ^{-1}}z) =\bar{\lambda }^{\gamma }b(\bar{\lambda }^{\gamma ^{-1}}z) - c(\gamma ,\bar{\lambda }^{\gamma ^{-1}}z)
	\end{align*}
for $\nu ^\hash\textit{-a.e.}\; z\in \bar{Z}_f$. If we replace $z$ by $\bar{\lambda }^{\gamma }z$ in the last equation we see that
	\begin{equation}
	\label{eq:coboundary}
c(\gamma ,\cdot )=\bar{\lambda }^\gamma \circ b-b\circ \bar{\lambda }^\gamma \quad \nu ^\hash\textit{-a.e.},\enspace \textup{for every}\enspace \gamma \in\Gamma .
	\end{equation}
In other words, the cocycle $c\colon \Gamma \times \bar{Z}_f\to K_f$ is a coboundary $(\textup{mod}\; \nu ^\hash)$ with Borel cobounding function $b\colon \bar{Z}_f\rightarrow B_1(K_f)$.

\smallskip Let $\theta _2\colon \tilde{Z}_f\rightarrow \tilde{Z}_f$ be the bijection given by
	\begin{displaymath}
\theta _2(z,v)=(z,v-b(z))
	\end{displaymath}
for every $(z,v)\in \tilde{Z}_f$, and put $\tilde{\nu }=(\theta _2)_*\tilde{\nu }^{(1)}=(\theta _2\circ \theta _1)_*\nu $. Then $\pi ^{(1)}_*\tilde{\nu }= \pi ^{(1)}_*\tilde{\nu }^{(1)} = \nu ^\hash$. We set
	\begin{equation}
	\label{eq:theta}
\theta = \theta _2\circ\theta _1 \colon \bar{Y}_f \to \tilde{Z}_f
	\end{equation}
and obtain that
	\begin{displaymath}
\theta (w) = \bigl(\bar{\rho }^fw, w-\zeta \circ \bar{\rho }^f(w) - b\circ \bar{\rho }^f(w)\bigr)\enspace \textup{for every}\enspace w\in \bar{Y}_f.
	\end{displaymath}
\smallskip For $\nu \textit{-a.e.}\; w\in \bar{Y}_f$ we have that
	\begin{equation}
	\label{eq:equi-2}
	\begin{aligned}
\tilde{\lambda }^\gamma \circ \theta (w) &= \bigl(\bar{\lambda }^\gamma \circ \bar{\rho }^f(w), \bar{\lambda }^\gamma w - \bar{\lambda }^\gamma \circ \zeta \circ \bar{\rho }^f(w) - \bar{\lambda }^\gamma \circ b\circ \bar{\rho }^f(w)\bigr)
	\\
&=\bigl(\bar{\rho }^f(\bar\lambda ^\gamma w), \bar{\lambda }^\gamma w + c(\gamma ,\bar{\rho }^fw) - \zeta \circ \bar{\rho }^f\circ \bar{\lambda }^\gamma (w)
	\\
& \hspace{20mm}- c(\gamma ,\bar{\rho }^fw) - b\circ \bar{\rho }^f\circ \bar{\lambda }^\gamma (w)\bigr)\qquad \textup{(by \eqref{eq:cocycle} and \eqref{eq:coboundary})}
	\\
&=\bigl(\bar{\rho }^f(\bar{\lambda }^\gamma w), \bar{\lambda }^\gamma w - \zeta \circ \bar{\rho }^f (\bar{\lambda }^\gamma w) - b\circ \bar{\rho }^f(\bar{\lambda }^\gamma w)\bigr) =\theta \circ \bar{\lambda }^\gamma (w),
	\end{aligned}
	\end{equation}
which proves that the $\Gamma $-actions $\bar{\lambda }$ on $(\bar{Y}_f,\nu )$ and $\tilde{\lambda }$ on $(\tilde{Z}_f,\tilde{\nu })$ are measurably conjugate, as claimed in (3).
	\end{proof}

In the next section we show that, if $\Gamma $ is amenable and $\lambda _f$ has finite and completely positive entropy, there exist unique $\bar{\lambda }$-invariant Borel probability measures $\nu _f$ on $\bar{Y}_f$ and $\nu _f^\hash $ on $\bar{Z}_f$ such that the principal algebraic $\Gamma $-action $\lambda _f$ on $(X_f,\mu _f)$ is measurably conjugate to the $\Gamma $-actions $\bar{\lambda }_{\bar{Y}_f}$ and $\bar{\lambda }_{\bar{Z}_f}$ on $(\bar{Y}_f, \nu _f)$ and $(\bar{Z}_f, \nu _f^\hash)$, respectively (cf. Theorem \ref{t:symbolic}).

\section{Symbolic representation of intrinsically ergodic principal algebraic actions}
	\label{s:symbolic}

Throughout this section we assume that $\Gamma $ is a countably infinite discrete amenable group and that $f\in\mathbb{Z}\Gamma $ is nonzero. We denote by $\mu _f$ the normalized Haar measure on $X_f$ and define $\bar{Y}_f\subset W_f\subset \ell ^\infty (\Gamma ,\mathbb{R})$ and $\bar{Z}_f=\bar{\rho }^f(\bar{Y}_f)\subset \ell ^\infty (\Gamma ,\mathbb{Z})$ as in \eqref{eq:Wf} -- \eqref{eq:Yf-Zf}.

	\begin{lemma}
	\label{l:zero divisor}
The principal algebraic $\Gamma $-action $\lambda _f$ on $X_f$ has infinite topological entropy if and only if $f$ is a left zero divisor in $\mathbb{R}\Gamma $, i.e., if and only if there exists a nonzero $g\in \mathbb{R}\Gamma $ with $fg=0$.
	\end{lemma}

	\begin{proof}
This is a special case of \cite{Chung-Li}*{Theorem 4.11}. For later reference we include here an explicit proof of the fact that $\htop(\lambda _f)=\infty $ if $f$ is a left zero divisor in $\mathbb{R}\Gamma $.

We embed $\mathbb{R}\Gamma $ in $\ell ^\infty (\Gamma ,\mathbb{R})$ in the obvious manner by identifying each $h=\sum_{\gamma \in \Gamma }h_\gamma \cdot \gamma \in \mathbb{R}\Gamma $ with $(h_\gamma )_{\gamma \in \Gamma }\in \ell ^\infty (\Gamma ,\mathbb{R})$.

If $f \in \mathbb{Z}\Gamma $ is a left zero divisor in $\mathbb{R}\Gamma $ we choose $g = \sum_{\gamma \in \Gamma }g_\gamma \cdot \gamma \in\mathbb{R}\Gamma $ with $1_\Gamma \in \textup{supp}(g)$ and $fg=0$. Then $\bar{\rho }^fg^*= g^*f^* = 0$ (cf. \eqref{eq:convolutions}), and hence $\bar{\rho }^f(cg^*) =cg^*f^* = 0$ for every $c\in \mathbb{R}$. This shows that $cg^*\in W_f$ and $\eta (cg^*)\in X_f$ for every $c\in \mathbb{R}$ (cf. \eqref{eq:Wf}).

If $I\subset \mathbb{R}$ is the open interval $(-\frac{1}{2\|g\|_\infty }, \frac{1}{2\|g\|_\infty })$, then the elements $\eta (cg^*)\in X_f$, $0\ne c\in I$, are all distinct with identical supports $E=\textup{supp}(g^*)=\textup{supp}(g)^{-1}$.

Choose a maximal set $D\subset \Gamma $ such that the translates $\{\delta E\mid \delta \in D\}$ are disjoint. We claim that $DEE^{-1}= \bigcup_{\delta \in D}\delta EE^{-1} = \Gamma $. Indeed, if $DEE^{-1}\ne \Gamma $, then there exists a $\gamma \in \Gamma $ which is not equal to $\delta \gamma '\gamma ''^{-1}$ for any $\delta \in D$ and $\gamma ',\gamma ''\in E$. Then $\gamma E \cap \delta E = \varnothing $ for every $\delta \in D$, which contradicts the maximality of $D$. This proves the last claim.

Since the sets $\delta E,\,\delta \in D$, are disjoint, we obtain for every $z=(z_\delta )_{\delta \in D}\in I^D$, a point $\tilde{z}\in W_f \cap (-\frac12, \frac12)^\Gamma $ which coincides on each $\delta E$ with $z_\delta \bar{\lambda }^{\delta }g^*$. This shows that the restriction of $X_f$ to its coordinates in $DE$ contains --- in essence --- a Cartesian product of the form $I^D$. Since $(DE)E^{-1}=\Gamma $ and $E^{-1}$ is finite, $\lambda _f$ must have infinite topological entropy on $X_f$.
	\end{proof}

	\begin{proposition}
	\label{p:kernel}
If the principal algebraic $\Gamma $-action $\lambda _f$ on $X_f$ has finite topological entropy, then the restriction $\bar{\lambda }_C$ of $\bar{\lambda }$ to every weak$^*$-closed, bounded, $\bar{\lambda }$-invariant subset $C\subset K_f$ has topological entropy zero {\textup(}cf. \eqref{eq:Kf}{\textup)}.
	\end{proposition}

For the proof of Proposition \ref{p:kernel} we need a lemma.

	\begin{lemma}
	\label{L-small entropy}
Let $\tau \colon \Gamma \to \textup{Aut}(X)$ be an action of a countably infinite discrete amenable group $\Gamma $ by continuous automorphisms of a compact metrizable group $X$ such that $\htop(\tau )<\infty$. Then there exists, for every $\varepsilon >0$, a neighbourhood $U$ of $1_X$ in $X$ such that the topological entropy $\htop(\tau _C)$ of the restriction of $\tau $ to any closed $\tau $-invariant subset $C\subset U$ is less than $\varepsilon $.
	\end{lemma}

	\begin{proof}
Choose a compatible left translation invariant metric $\mathsf{d}$ on $X$ (i.e., $\mathsf{d}(x,y) = \mathsf{d}(zx,zy)$ for all $x,y,z\in X$). For every nonempty finite subset $F \Subset \Gamma $, put
	\begin{displaymath}
\mathsf{d}_F(x,y)=\max\nolimits_{\gamma \in F}\mathsf{d}(\tau ^\gamma x,\tau ^\gamma y), \enspace x,y\in X.
	\end{displaymath}
For each $\zeta >0$ we denote by $\textup{sep}(X, \mathsf{d}, \zeta )$ the maximal cardinality of subsets $Z\subset X$ which are $(\mathsf{d}, \zeta )$-separated in the sense that $\mathsf{d}(y, z)\ge \zeta $ for all distinct $y, z\in Z$.

\smallskip Take a left F{\o}lner sequence $(F_n)_{n\ge1}$ for $\Gamma $, i.e., a sequence of nonempty finite sets $F_n\subset \Gamma $ with $\lim_{n\to\infty }\frac{|\gamma F_n \cap F_n|}{|F_n|}=0$ for every $\gamma \in \Gamma $. Then there exists, for every $\varepsilon >0$, some $\zeta >0$ such that
	\begin{displaymath}
\liminf_{n\to \infty} \frac{1}{|F_n|}\log \textup{sep}(X, \mathsf{d}_{F_n}, \zeta )\ge \htop(\tau )-\varepsilon/2.
	\end{displaymath}
For large enough $n$, take a $(\mathsf{d}_{F_n}, \zeta )$-separated set $X_n\subset X$ such that $\frac{1}{|F_n|}\log |X_n|\ge \htop(\tau )-\varepsilon$.

\smallskip Put $U = \{x\in X\mid \mathsf{d}(x, 1_X)<\zeta /10\}$. Let $Y\subset U$ be closed and $\Gamma $-invariant, and let $\tau _Y$ be the restriction of $\tau $ to $Y$. In order to show that $\htop(\tau _Y)\le \varepsilon$ it suffices to show that
	\begin{align}
	\label{E-entropy small}
\limsup_{n\to \infty} \frac{1}{|F_n|}\log \textup{sep}(Y, \mathsf{d}_{F_n}, \delta)\le \varepsilon
	\end{align}
whenever $0<\delta<\zeta /10$. In order to verify \eqref{E-entropy small} we choose, for each $n$, a $(\mathsf{d}_{F_n}, \delta)$-separated set $Y_n\subset Y$ of cardinality $|Y_n|=\textup{sep}(Y, \mathsf{d}_{F_n}, \delta)$. When $n$ is large enough, then $|X_nY_n|=|X_n|\cdot |Y_n|$ and $X_nY_n$ is $(\mathsf{d}_{F_n}, \delta)$-separated: indeed,
	\begin{displaymath}
\mathsf{d}_{F_n}(xy, xz)=\mathsf{d}_{F_n}(y, z)\ge \delta,
	\end{displaymath}
for $x\in X_n$ and distinct $y, z\in Y_n$, whereas
	\begin{align*}
\mathsf{d}_{F_n}(x_1y, x_2z)&\ge \mathsf{d}_{F_n}(x_1, x_2)-\mathsf{d}_{F_n}(x_1y, x_1)-\mathsf{d}_{F_n}(x_2z, x_2)
	\\
&=\mathsf{d}_{F_n}(x_1, x_2)-\mathsf{d}_{F_n}(y, 1_X)-\mathsf{d}_{F_n}(z, 1_X) \ge \zeta -\zeta /10-\zeta /10\ge \delta
	\end{align*}
for $y, z\in Y_n$ and distinct $x_1, x_2\in X_n$.

\smallskip Then
	\begin{align*}
\htop(\tau )&\ge \limsup_{n\to \infty}\tfrac{1}{|F_n|}\log \textup{sep}(X, \mathsf{d}_{F_n}, \delta) \ge \limsup_{n\to \infty}\tfrac{1}{|F_n|}\log (|X_n|\cdot |Y_n|)
	\\
&\ge \htop(\tau )-\varepsilon+\limsup_{n\to \infty}\tfrac{1}{|F_n|}\log \textup{sep}(Y, \mathsf{d}_{F_n}, \delta),
	\end{align*}
which implies \eqref{E-entropy small}.
	\end{proof}

	\begin{proof}[Proof of Proposition \ref{p:kernel}] Since the $\Gamma $-actions $\bar{\lambda }_{B_r(K_f)},\, r>0$, are all conjugate to each other, $\htop(\bar{\lambda }_{B_r(K_f)})$ is the same for all $r>0$.

\smallskip For $0<r<1/2$, the map $\eta \colon \ell^\infty(\Gamma, \mathbb{R})\rightarrow \mathbb{T}^\Gamma $ in \eqref{eq:eta} embeds $B_r(K_f)$ injectively as a closed $\Gamma $-invariant subset of $X_f$. If $U\subset X_f$ is any open neighbourhood of $1_{X_f}$, then $\eta (B_r(K_f))\subset U$ for all sufficiently small $r>0$.

Let $C\subset K_f$ be a weak$^*$-closed, bounded, $\bar{\lambda }$-invariant subset. Then $C\subset B_r(K_f)$ for some $r>0$, and $\htop(\bar{\lambda }_C)\le \htop(\bar{\lambda }_{B_r(K_f)}) = \htop(\bar{\lambda }_{B_{r'}(K_f)})$ for every $r'>0$.

Let $\varepsilon>0$. By Lemma~\ref{L-small entropy} there is some neighbourhood $U$ of $1_{X_f}$ in $X_f$ such that for any closed $\Gamma$-invariant subset $Y$ of $X_f$ contained in $U$ the restriction $(\lambda _f)_Y$ of $\lambda _f$ to $Y$ has entropy $\le \varepsilon $. If $r'>0$ is small enough, $\eta(B_{r'}(K_f))\subset U$, so that
	\begin{displaymath}
\htop(\bar{\lambda }_C)\le \htop(\bar{\lambda }_{B_{r}(K_f)}) = \htop(\bar{\lambda }_{B_{r'}(K_f)}) = \htop((\lambda _f)_{\eta (B_{r'}(K_f))}) \le \varepsilon .
	\end{displaymath}
As $\varepsilon>0$ is arbitrary, we conclude that $\htop(\bar{\lambda }_C)=0$.
	\end{proof}

	\begin{proposition}[cf. \cite{Leiden}*{Proposition 8.7}]
	\label{p:product}
Let $Y_1,Y_2$ be compact metrizable spaces, and let $\tau _1,\tau _2$ be continuous actions of a countably infinite discrete amenable group $\Gamma $ on $Y_1$ and $Y_2$ such that the topological entropy $\htop(\tau _2)$ of $\tau _2$ is equal to zero. We write $\pi ^{(i)}\colon Y_1\times Y_2\rightarrow Y_i$ for the two coordinate projections. If $\mu $ is a $(\tau _1\negthinspace \times \negthinspace \tau _2)$-invariant Borel probability measure on $Y_1\times Y_2$ we set $\mu _i=\pi ^{(i)}_*\mu $. Then $h_\mu (\tau _1\negthinspace \times \negthinspace \tau _2)=h_{\mu _1}(\tau _1)$.
	\end{proposition}

	\begin{proof}
Let $\mathcal{P}$ and $\mathcal{Q}$ be finite Borel partitions of $Y_1$ and $Y_2$, respectively, and set $\tilde{\mathcal{P}}=\{P\times Y_2\mid P\in \mathcal{P}\}$ and $\tilde{\mathcal{Q}}=\{Y_1\times Q\mid Q\in \mathcal{Q}\}$. If $(F_n)$ is a left F{\o}lner sequence in $\Gamma $, then
	\begin{align*}
h_\mu (\tau _1\negthinspace \times \negthinspace \tau _2,\tilde{\mathcal{P}}\vee \tilde{\mathcal{Q}}) & = \lim_{n\to\infty } \frac{1}{|F_n|}H_\mu \bigl(\textstyle\bigvee _{\gamma \in F_n}(\tau _1\negthinspace \times \negthinspace \tau _2)^{\gamma ^{-1}}(\tilde{\mathcal{P}}\vee \tilde{\mathcal{Q}})\bigr)
	\\
&= \lim_{n\to\infty } \frac{1}{|F_n|}H_\mu \bigl(\textstyle\bigvee _{\gamma \in F_n}\tau _1^{\gamma ^{-1}}(\tilde{\mathcal{P}})\vee \textstyle\bigvee _{\gamma \in F_n}\tau _2^{\gamma ^{-1}}(\tilde{\mathcal{Q}})\bigr)
	\\
&= \lim_{n\to\infty } \frac{1}{|F_n|}\Bigl(H_\mu \bigl(\textstyle\bigvee _{\gamma \in F_n}\tau _1^{\gamma ^{-1}}(\tilde{\mathcal{P}})\bigr)
	\\
&\qquad \qquad \qquad + H_\mu \bigl(\textstyle\bigvee _{\gamma \in F_n}\tau _2^{\gamma ^{-1}}(\tilde{\mathcal{Q}}) \bigm| \textstyle\bigvee _{\gamma \in F_n}\tau _1^{\gamma ^{-1}}(\tilde{\mathcal{P}})\bigr)\Bigr)
	\\
&\le \lim_{n\to\infty }\frac{1}{|F_n|} \Bigl(H_\mu \bigl(\textstyle \bigvee_{\gamma \in F_n} \tau _1^{\gamma ^{-1}}(\tilde{\mathcal{P}})\bigr) + H_\mu \bigl(\textstyle\bigvee _{\gamma \in F_n}\tau _2^{\gamma ^{-1}}(\tilde{\mathcal{Q}})\bigr)\Bigr)
	\\
&=\lim_{n\to\infty }\frac{1}{|F_n|} \Bigl(H_{\mu_1} \bigl(\textstyle \bigvee_{\gamma \in F_n} \tau _1^{\gamma ^{-1}}(\mathcal{P})\bigr) + H_{\mu _2}\bigl(\textstyle\bigvee _{\gamma \in F_n}\tau _2^{\gamma ^{-1}}(\mathcal{Q})\bigr)\Bigr)
	\\
&\le h_{\mu _1}(\tau _1,\mathcal{P}) + \htop(\tau _2) = h_{\mu _1}(\tau _1,\mathcal{P})
	\end{align*}
by the variational principle \cite{KL16}*{Theorem 9.48}. By taking the supremum over all finite partitions $\mathcal{P}$ and $\mathcal{Q}$ we obtain that $h_\mu (\tau _1\negthinspace \times \negthinspace \tau _2)\linebreak[0]\le h_{\mu _1}(\tau _1)$. The reverse inequality $h_{\mu _1}(\tau _1)\le h_\mu (\tau _1\negthinspace \times \negthinspace \tau _2)$ is obvious.
	\end{proof}

	\begin{proposition}[cf. \cite{Leiden}*{Corollary 8.9}]
	\label{p:no-drop}
Suppose that the principal algebraic $\Gamma $-action $\lambda _f$ on $X_f$ has finite topological entropy. Then the following is true.
	\begin{enumerate}
	\item
For every $\bar{\lambda }$-invariant Borel probability measure $\nu $ on $\bar{Y}_f$, the probability measure $\nu ^\hash = \bar{\rho }^f_*\nu $ on $\bar{Z}_f$ is $\bar{\lambda }$-invariant, and $h_\nu (\bar{\lambda }_{\bar{Y}_f})=h_{\nu ^\hash}(\bar{\lambda }_{\bar{Z}_f})$.
	\item
$\htop(\bar{\lambda }_{\bar{Y}_f})=\htop(\bar{\lambda }_{\bar{Z}_f})$.
	\end{enumerate}
	\end{proposition}

	\begin{proof}
Since $\bar{\rho }^f$ induces a continuous surjective $\bar{\lambda }$-equivariant map from $\bar{Y}_f$ to $\bar{Z}_f$, $\htop(\bar{\lambda }_{\bar{Z}_f})\le \htop(\bar{\lambda }_{\bar{Y}_f})$ and $h_{\nu ^\hash}(\bar{\lambda }_{\bar{Z}_f}) \le h_\nu (\bar{\lambda }_{\bar{Y}_f})$ for every $\bar{\lambda }$-invariant Borel probability measure $\nu $ on $\bar{Y}_f$.

By applying the Propositions \ref{p:kernel} and \ref{p:product} with $Y_1=\bar{Z}_f$, $Y_2=B_2(K_f)$, $\tau _1 = \bar{\lambda }_{\bar{Z}_f}$, $\tau _2= \bar{\lambda }_{B_2(K_f)}$, $\mu = \tilde{\nu }$, and $\mu _1=\nu ^\hash = \bar{\rho }^f_*\nu $, we obtain that $h_\nu (\bar\lambda _{\bar{Y}_f}) = h_{\tilde{\nu }}(\tilde{\lambda })= h_{\nu ^\hash}(\bar\lambda _{\bar{Z}_f})$. This proves (1).

\smallskip (2): The variational principle \cite{KL16}*{Theorem 9.48} implies that
	\begin{displaymath}
\htop(\bar{\lambda }_{\bar{Y}_f})=\sup_\mu h_\mu (\bar{\lambda }_{\bar{Y}_f})=\sup_\mu h_{\bar{\rho }^f_*\mu }(\bar{\lambda }_{\bar{Z}_f})\le \htop(\bar{\lambda }_{\bar{Z}_f}),
	\end{displaymath}
where the supremum is taken over the set of $\bar{\lambda }$-invariant Borel probability measures $\mu $ on $\bar{Y}_f$. Since the opposite inequality $\htop(\bar{\lambda }_{\bar{Z}_f})\le \htop(\bar{\lambda }_{\bar{Y}_f})$ is trivially satisfied, this completes the proof of the proposition.
	\end{proof}

Proposition \ref{p:kernel} yields a strengthening of Proposition \ref{p:no-drop} for $\bar{\lambda }$-invariant probability measures $\nu $ on $\bar{Y}_f$ with completely positive entropy. We use the same notation as in the Propositions \ref{p:conjugacy} and \ref{p:no-drop}.

	\begin{corollary}
	\label{c:conjugacy}
Suppose that the principal algebraic $\Gamma $-action $\lambda _f$ on $X_f$ has finite topological entropy. Then the $\Gamma $-actions $\bar{\lambda }_{\bar{Y}_f}$ on $(\bar{Y}_f,\nu )$ and $\bar{\lambda }_{\bar{Z}_f}$ on $(\bar{Z}_f, \nu ^\hash)$ are measurably conjugate for every $\bar{\lambda }$-invariant Borel probability measure $\nu $ on $\bar{Y}_f$ with completely positive entropy.
	\end{corollary}

	\begin{proof}
As in the proof of Proposition \ref{p:conjugacy} we define $\theta _1\colon \bar{Y}_f \to \tilde{Z}_f$ by \eqref{eq:theta-1} and set $\tilde{\nu }^{(1)} = (\theta _1)_*\nu $ and $\tilde{\nu } = (\theta _2)_*\tilde{\nu }^{(1)} = \theta _*\nu $. Then $\tilde{\nu }$ is $\tilde{\lambda }$-invariant by \eqref{eq:equi-2}, and $\pi ^{(1)}_*\tilde{\nu }= \nu ^\hash$. We write $\pi ^{(2)}\colon \tilde{Z}_f\to K_f$ for the second coordinate projection, denote by $\xi _f= \pi ^{(2)}_*\tilde{\nu }$ the projection of $\tilde{\nu }$ onto $K_f$, and note that the $\Gamma $-action $\tilde{\lambda }$ on $(\tilde{Z}_f, \tilde{\nu })$ has the zero-entropy $\Gamma $-action $\bar{\lambda }$ on $(K_f, \xi _f)$ as a factor (cf. Proposition \ref{p:kernel}). Since $\tilde{\lambda }$ on $(\tilde{Z}_f, \tilde{\nu })$ is measurably conjugate to $\bar{\lambda }$ on $(\bar{Y}_f,\nu )$ and thus has completely positive entropy, we obtain a contradiction unless $\xi _f$ is concentrated in a single point.

Since $\xi _f$ is a point mass, the first coordinate projection $\pi ^{(1)}\colon (\tilde{Z}_f, \tilde{\nu }) \to (\bar{Z}_f, \nu ^\hash)$ is injective $(\textup{mod}\;\tilde{\nu })$, and the $\Gamma $-actions $\tilde{\lambda }$ on $(\tilde{Z}_f, \tilde{\nu })$ and $\bar{\lambda }$ on $(\bar{Z}_f, \nu ^\hash)$ are conjugate. This proves that the $\Gamma $-actions $\bar{\lambda }$ on $(\bar{Y}_f,\nu )$ and on $(\bar{Z}_f,\nu ^\hash)$ are measurably conjugate.
	\end{proof}

Having discussed the relation between $\bar{\lambda }$-invariant probability measures on $\bar{Y}_f$ and $\bar{Z}_f$ we turn to the corresponding question for measures on $\bar{Y}_f$ and their images under $\eta $.

	\begin{lemma}
	\label{l:comparison32}
There exists a unique $\bar{\lambda }$-invariant Borel probability measure $\nu _f$ on $\bar{Y}_f$ with $\eta _*\nu _f=\mu _f$, and the map $\eta \colon \bar{Y}_f \to X_f$ induces a conjugacy of the $\Gamma $-actions $\bar{\lambda }$ on $(\bar{Y}_f,\nu _f)$ and $\lambda _f$ on $(X_f,\mu _f)$.
	\end{lemma}

	\begin{proof}
Let $\nu $ be a $\bar{\lambda }$-invariant Borel probability measure on $\bar{Y}_f$ such that $\eta _*\nu =\mu _f$. If $\nu (\bar{Y}_f\smallsetminus Y_f)>0$, the set
	\begin{displaymath}
V =\{y\in \bar{Y}_f \mid y_{1_\Gamma }=1\}
	\end{displaymath}
must have positive $\nu $-measure, which implies that the closed subgroup
	\begin{displaymath}
H=\{x\in X_f\mid x_{1_\Gamma }=0\} \supset \eta (V)
	\end{displaymath}
has positive $\mu _f$-measure. We set
	\begin{displaymath}
K=\pi _{1_\Gamma }(X_f)\subset \mathbb{T},
	\end{displaymath}
observe that $K$ is a closed subgroup of $\mathbb{T}$, and denote by $\mu _K$ the normalized Haar measure of $K$. Since $\mu _K(\{t\})=\mu _K(\{0\})=\mu _f(H) > 0$ for every $t\in K$, the group $K$ must be finite, which implies that $X_f \subset K^\Gamma $ and hence $Y_f=\bar{Y}_f$, contrary to our assumption that $\nu (\bar{Y}_f\smallsetminus Y_f)>0$. It follows that $\nu (B)=\nu (B\cap Y_f)= \mu _f(\eta (B))$ for every Borel set $B\subset \bar{Y}_f$, as claimed. Hence the map $\eta \colon \bar{Y}_f \to X_f$ induces a measure space isomorphism from $(\bar{Y}_f,\nu _f)$ to $(X_f,\mu _f)$ which carries the $\Gamma $-action $\bar{\lambda }_f$ on $(\bar{Y}_f,\nu _f)$ to $\lambda _f$ on $(X_f,\mu _f)$.
	\end{proof}

	\begin{theorem}
	\label{t:comparison2}
Let $\Gamma $ be a countably infinite discrete amenable group, $f\in \mathbb{Z}\Gamma $, and assume that the principal algebraic $\Gamma $-action $\lambda _f$ on $X_f$ has finite topological entropy. Then $\htop(\bar{\lambda }_{\bar{Z}_f}) = \htop(\bar{\lambda }_{\bar{Y}_f})=\htop(\lambda _f)$.
	\end{theorem}

We start the proof of Theorem \ref{t:comparison2} with four lemmas. For any finite subset $F\subset \Gamma $ containing $1_\Gamma $ and any $Q\subset \Gamma $ we put
	\begin{equation}
	\label{eq:interior}
\textup{Int}_FQ=\{\gamma \in \Gamma \mid \gamma F\subset Q\}.
	\end{equation}

	\begin{lemma}[\cite{Leiden}*{Lemma 6.4}]
	\label{l:induction}
Let $V$ be a finite dimensional vector space over $\mathbb{R}$, and let $k>\dim V$. Let $\phi _1,\dots ,\phi _k$ be affine functions on $V$ and $b_1,\dots ,b_k \in \mathbb{R}$. Then there exist $a_1,\dots ,a_k\in\{0,1\}$ such that $\bigcap_{j=1}^k W_j (a_j)= \varnothing $, where
	\begin{displaymath}
W_j (a_j)=
	\begin{cases}
\{x\in V\mid \phi _j(x) < b_j\}&\textup{if}\enspace a_j=0,
	\\
\{x\in V\mid \phi _j(x)\ge b_j\}&\textup{if}\enspace a_j=1.
	\end{cases}
	\end{displaymath}
	\end{lemma}

	\begin{lemma}
	\label{l:dimension}
Suppose that $f\in \mathbb{R}\Gamma $ satisfies that $\htop(\lambda _f)<\infty $, and that $1_\Gamma \in E=\textup{supp}(f)$. Let $Q\Subset \Gamma $. For every nonzero $v\in \mathbb{R}\Gamma $, the product $v\cdot f^*$ is nonzero {\textup(}since $f$ is not a left zero divisor{\textup)}, and the restriction of $v\cdot f^*$ to $\textup{Int}_EQ$ depends only on the restriction $\pi _Q(v)$ of $v$ to $Q$: for every $v,v'\in \mathbb{R}\Gamma $ with $\pi _Q(v)=\pi _Q(v')$, $\pi _{\textup{Int}_EQ}(v\cdot f^*)=\pi _{\textup{Int}_EQ}(v'\cdot f^*)$. Since the map $v \mapsto \bar{\rho }^fv = v\cdot f^*$ from $\mathbb{R}\Gamma $ to $\mathbb{R}\Gamma $ in \eqref{eq:convolutions} induces an injective map from $\mathbb{R}^Q$ to $\mathbb{R}^{QE^{-1}}$, the linear space
	\begin{equation}
	\label{eq:VQ}
V_Q=\bigl\{v=(v_\gamma )_{\gamma \in Q}\in \mathbb{R}^Q \mid \pi _{\textup{Int}_EQ}(v\cdot f^*) =0\bigr\},
	\end{equation}
has dimension $\dim V_Q \le |QE^{-1}\smallsetminus \textup{Int}_EQ|$\enspace {\textup(}cf. \eqref{eq:interior}{\textup)}.
	\end{lemma}

	\begin{proof}
Since $\dim \,(\{w\in \mathbb{R}^{QE^{-1}}\mid \pi _{\textup{Int}_EQ}(w)=0\}) = |QE^{-1}\smallsetminus \textup{Int}_EQ|$ and the map $\mathbb{R}^Q\to \mathbb{R}^{QE^{-1}}$ induced by $\bar{\rho }^f$ is injective, $\dim V_Q\le |QE^{-1}\smallsetminus \textup{Int}_EQ|$, as claimed.
	\end{proof}

For the next lemma we recall that a family of subsets $\mathcal{Z}$ of a finite set $Z$ is said to \emph{scatter} a set $J\subset Z$ if $\mathcal{Z}\cap J= \{C\cap J\mid C\in \mathcal{Z}\} = \mathcal{P}(J)$, the set of all subsets of $J$.

	\begin{lemma}[Sauer-Perles-Shelah \cite{Pajor}, \cite{Sauer}*{Theorem 1}, \cite{Shelah}]
	\label{l:SPS}
Let $Z$ be a finite set with cardinality $n\ge1$ and let $\mathcal{Z}$ be a collection of subsets of $Z$. If $|\mathcal{Z}| > \sum_{i=0}^{k-1}\binom{|Z|}{i}$ for some $k\in \{1,\dots ,|Z|\}$, then $\mathcal{Z}$ scatters a subset $J\subset Z$ of size $k$.
	\end{lemma}

	\begin{proof}
For a proof see, e.g., \cite{Wiki}.
	\end{proof}

	\begin{lemma}[\cite{CCL}*{Lemma A.1}]
	\label{l:Stirling}
Let $0<\beta <1/2$. Then there exist $\kappa =\kappa (\beta )>0$ and $m_0=m_0(\beta )\in \mathbb{N} = \{1,2,\dots \}$ with
	\begin{displaymath}
\lim\nolimits_{\beta \to0} \kappa (\beta )=0,
	\end{displaymath}
such that
	\begin{displaymath}
\sum\nolimits_{i=0}^{\lfloor \beta m\rfloor} \binom {m}{i} \le e^{\kappa m}
	\end{displaymath}
for all $m\in \mathbb{N}$ with $m\ge m_0$.
	\end{lemma}

	\begin{proof}[Proof of Theorem \ref{t:comparison2}]
For $\delta \in \Gamma $, the spaces $X_f, Y_f$ do not change if we replace $f = \sum_{\gamma \in \Gamma }f_\gamma \cdot \gamma $ by $\delta f = \sum_{\gamma \in \Gamma }f_\gamma \cdot \delta \gamma $, and $Z_{\delta f}=\bar{\rho }^\delta (Z_f)$. For the proof of this theorem we may therefore assume without loss of generality that $1_\Gamma \in \textup{supp}(f)$.

Since the continuous shift-equivariant map $\eta \colon \ell ^\infty (\Gamma ,\mathbb{R})\to \mathbb{T}^\Gamma $ in \eqref{eq:eta} sends $\bar{Y}_f$ onto $X_f$, we know that $\htop(\bar{\lambda }_{\bar{Y}_f})\ge \htop(\lambda _f)$. It will suffice to show that $\htop(\bar{\lambda }_{\bar{Y}_f})\le \htop(\lambda _f)$.

Denote by $\mathsf{d}_\mathbb{I}(s,t) = |s-t|$ the Euclidean metric on the closed unit interval $\mathbb{I}=[0,1]$ and by $\mathsf{d}_\mathbb{T}$ the metric on $\mathbb{T}$ given by
	\begin{displaymath}
\mathsf{d}_\mathbb{T}(s+\mathbb{Z},t+\mathbb{Z}) = \min_{k\in \mathbb{Z}}|s-t-k|.
	\end{displaymath}
For any $F\Subset \Gamma $ we define continuous pseudometrics $\mathsf{d}_\mathbb{I}^{(F)}$ and $\mathsf{d}_\mathbb{T}^{(F)}$ on $\bar{Y}_f$ and $X_f$, respectively, by
	\begin{gather*}
\mathsf{d}_\mathbb{I}^{(F)}(y, y') \coloneqq \max_{\gamma \in F^{-1}} \mathsf{d}_\mathbb{I}(y_\gamma , y_\gamma '),\enspace y,y'\in \bar{Y}_f,
	\\
\mathsf{d}_\mathbb{T}^{(F)}(x, x') \coloneqq \max_{\gamma \in F^{-1}} \mathsf{d}_\mathbb{T}(x_\gamma , x_\gamma '),\enspace x,x'\in X_f.
	\end{gather*}
For every $\varepsilon >0$ we denote by $\textup{sep}(\bar{Y}_f,\mathsf{d}_\mathbb{I}^{(F)}, \varepsilon )$ and $\textup{sep}(X_f,\mathsf{d}_\mathbb{T}^{(F)}, \varepsilon )$ the maximal cardinalities of $(\mathsf{d}_\mathbb{I}^{(F)},\varepsilon )$-separated subsets of $\bar{Y}_f$ and $(\mathsf{d}_\mathbb{T}^{(F)},\varepsilon )$-separated subsets of $X_f$, respectively.

\smallskip Let $(F_n)$ be a left F{\o}lner sequence of $\Gamma $. By \cite{D06}*{Proposition 2.3} we have
	\begin{gather*}
\htop(\bar{\lambda }_{\bar{Y}_f}) = \sup_{\varepsilon >0} \limsup_{n \to \infty } \frac {\log \textup{sep}(\bar{Y}_f,\mathsf{d}_\mathbb{I}^{(F_n)}, \varepsilon )}{|F_n|},
	\\
\htop(\lambda _f) = \sup_{\varepsilon >0} \limsup_{n \to \infty } \frac {\log \textup{sep}(X_f,\mathsf{d}_\mathbb{T}^{(F_n)}, \varepsilon )}{|F_n|}.
	\end{gather*}
Assume that $\htop(\bar{\lambda }_{\bar{Y}_f}) > \htop(\lambda _f)$. Then we can find some $0<\varepsilon <\frac{1}{\max(10,2\|f\|_1)}$ and $c > 0$ such that, passing to a subsequence of $(F_n)$ if necessary, one has
	\begin{equation}
	\label{eq:separated}
\textup{sep}(\bar{Y}_f,\mathsf{d}_\mathbb{I}^{(F_n)}, \varepsilon ) \ge \textup{sep}(X_f,\mathsf{d}_\mathbb{T}^{(F_n)}, \varepsilon /3) \exp(c|F_n|)
	\end{equation}
for all $n\ge1$.

\smallskip We fix $n\ge1$ for the moment and choose a $(\mathsf{d}_\mathbb{I}^{(F_n)},\varepsilon )$-separated subset $\mathcal{W}_n\subset \bar{Y}_f$ with $|\mathcal{W}_n|=\textup{sep}(\bar{Y}_f,\mathsf{d}_\mathbb{I}^{(F_n)}, \varepsilon )$. Then $\mathcal{W}_n$ is $(\mathsf{d}_\mathbb{I}^{(F_n)},\varepsilon )$-spanning in $\bar{Y}_f$. Since $\eta ^{-1}(X_f)\cap [0,1)^\Gamma $ is dense in $\bar{Y}_f$, we may move some of the points in $\mathcal{W}_n$ by less than $\varepsilon /10$ in the pseudometric $\mathsf{d}_\mathbb{I}^{(F_n)}$, if necessary, and without changing notation, so that $\mathcal{W}_n\subset \bar{Y}_f \cap [0,1)^\Gamma $, while remaining $(\mathsf{d}_\mathbb{I}^{(F_n)},\varepsilon )$-spanning and $(\mathsf{d}_\mathbb{I}^{(F_n)},4\varepsilon /5)$-separated in $\bar{Y}_f$. Similarly, if $\mathcal{V}_n\subset X_f$ is a maximal $(\mathsf{d}_\mathbb{T}^{(F_n)},\varepsilon /3)$-separated subset in $X_f$, then
	\begin{displaymath}
X_f\subset \bigcup\nolimits_{x\in \mathcal{V}_n}B_\mathbb{T}^{(F_n)}(x,\varepsilon /3),
	\end{displaymath}
where $B_\mathbb{T}^{(F_n)}(x,\varepsilon /3)$ is the open $\mathsf{d}_\mathbb{T}^{(F_n)}$-ball in $X_f$ with centre $x$ and radius $\varepsilon /3$, and we can find, for every $n\ge1$, a point $z^{(n)}\in \mathcal{V}_n$ such that $|\eta (\mathcal{W}_n)\cap B_\mathbb{T}^{(F_n)}(z^{(n)},\varepsilon /3)|\ge \exp(c|F_n|)$ (cf. \eqref{eq:separated}). For every $n\ge1$ we set $\mathcal{W}_n'=\{y\in \mathcal{W}_n\mid \eta (y)\in B_\mathbb{T}^{(F_n)}(z^{(n)},\varepsilon /3)\}$ and denote by $\tilde{z}^{(n)}\in [0,1)^{F_n^{-1}}$ the unique point with $z_\gamma ^{(n)}= \tilde{z}_\gamma ^{(n)}\;(\textup{mod\,1})$ for every $\gamma \in F_n^{-1}$.

For every $y\in \mathcal{W}_n'$ there is a unique $\tilde{y}\in \{-1,0,1\}^{F_n^{-1}}$ such that $|y_\gamma -\tilde{y}_\gamma -\tilde{z}_\gamma ^{(n)}| < \varepsilon /3$ for every $\gamma \in F_n^{-1}$. We set
	\begin{displaymath}
G_y^+= \{\gamma \in F_n^{-1}\mid \tilde{y}_\gamma =1\},\qquad G_y^\circ = \{\gamma \in F_n^{-1}\mid \tilde{y}_\gamma =0\}, \qquad G_y^-= \{\gamma \in F_n^{-1}\mid \tilde{y}_\gamma =-1\}.
	\end{displaymath}
Since $\mathcal{W}_n'$ is $(\mathsf{d}_\mathbb{I}^{(F_n)},4\varepsilon /5)$-separated and $G_y^+\cup G_y^\circ \cup G_y^- = F_n^{-1}$, it is clear that $\tilde{y}\ne \tilde{y}'$ and hence $(G_y^+,G_y^-)\ne (G_{y'}^+,G_{y'}^-)$ for any $y\ne y'$ in $\mathcal{W}_n'$.

We recall that $1_\Gamma \in E=\textup{supp}(f)$ and define $\textup{Int}_EF_n^{-1}$ as in \eqref{eq:interior}. For any $y\in \mathcal{W}_n$, the restrictions of $y\cdot f^*$ and $y|_{F_n^{-1}}\cdot f^*$ to $\textup{Int}_EF_n^{-1}$ coincide; since $y\cdot f^* \in \ell ^\infty (\Gamma ,\mathbb{Z})$, this implies that $y|_{F_n^{-1}}\cdot f^*$ and $(y|_{F_n^{-1}}-\tilde{z}^{(n)})\cdot f^*$ have integral coordinates on $\textup{Int}_EF_n^{-1}$. Furthermore, since $|y_\gamma -\tilde{y}_\gamma -\tilde{z}^{(n)}_\gamma |<\varepsilon /3$ for every $\gamma \in F_n^{-1}$, we obtain that
	\begin{displaymath}
\bigl\|\bigl((y|_{F_n^{-1}}-\tilde{y}-\tilde{z}^{(n)}) \cdot f^*\bigr)\big|_{\textup{Int}_EF_n^{-1}}\bigr \|_\infty < \frac{\varepsilon }{3}\cdot \|f\|_1 <1,
	\end{displaymath}
so that
	\begin{displaymath}
v(y)\coloneqq y|_{F_n^{-1}}-\tilde{y}-\tilde{z}^{(n)} \in V_{F_n^{-1}}
	\end{displaymath}
for every $y\in \mathcal{W}_n'$ (cf. \eqref{eq:VQ}).

\smallskip Put $\mathscr{W}_n''=\{(G_y^+, G_y^-)\mid y\in \mathscr{W}_n'\}$, $\mathscr{W}_n^+=\{G_y^+\mid y\in \mathscr{W}_n'\}$ and $\mathscr{W}_n^-=\{G_y^-\mid y\in \mathscr{W}_n'\}$. Since $|\mathcal{W}_n''|= |\mathcal{W}_n'| \ge \exp(c|F_n|)$ it is clear that $\max (|\mathcal{W}_n^+|, |\mathcal{W}_n^-|) \ge \exp(c|F_n|/2)$.

Suppose that $|\mathcal{W}_n^+| \ge \exp(c|F_n|/2)$ for infinitely many $n\ge1$ (if $|\mathcal{W}_n^-| \ge \exp(c|F_n|/2)$ for infinitely many $n$ the proof is completely analogous). By passing to a subsequence we may assume that $|\mathcal{W}_n^+| \ge \exp(c|F_n|/2)$ for every $n\ge1$. By Lemma \ref{l:Stirling} there exists $\beta >0$ such that $\kappa = \kappa (\beta )< c/2$ and $\sum\nolimits_{i=0}^{\lfloor\beta |F_n|\rfloor}{\tbinom{|F_n|}{i}} < \exp(\kappa |F_n|) <\exp(c|F_n|/2) \le |\mathcal{W}_n^+| $ for every sufficiently large $n\ge1$. Lemma \ref{l:SPS} implies that $\mathcal{W}_n^+$ scatters a subset $J_n^+\subset F_n^{-1}$ of size $\ge \beta |F_n^{-1}|$.

\smallskip We are going to show that $\dim V_{F_n^{-1}} \ge |J_n^+|$ for infinitely many $n\ge1$, thereby contradicting Lemma \ref{l:dimension}. For this we define, for every $\gamma \in F_n^{-1}$, a linear functional $\phi _\gamma \colon V_{F_n^{-1}}\to \mathbb{R}$ by setting $\phi _\gamma (v) = v_\gamma $ for every $v\in V_{F_n^{-1}}$. For every $y\in \mathcal{W}_n'$ and $\gamma \in F_n^{-1}$, we have the following possibilities:
	\begin{gather*}
\gamma \in G_y^+\enspace \textup{and}\enspace \phi _\gamma (v(y)) + \tilde{z}_\gamma ^{(n)} = y_\gamma -1 < 0,
	\\
\gamma \in G_y^\circ \enspace \textup{and}\enspace 1 > \phi _\gamma (v(y)) + \tilde{z}_\gamma ^{(n)} = y_\gamma \ge 0,
	\\
\gamma \in G_y^- \enspace \textup{and}\enspace \phi _\gamma (v(y)) + \tilde{z}_\gamma ^{(n)} = y_\gamma +1 \ge 1.
	\end{gather*}
In particular, $\phi _\gamma (v(y))<-\tilde{z}_\gamma ^{(n)}$ if $\gamma \in G_y^+$, and $\phi _\gamma (v(y))\ge -\tilde{z}_\gamma ^{(n)}$ if $\gamma \in F_n^{-1}\smallsetminus G_y^+$.

\smallskip We can thus find, for every subset $H\subset J_n^+$, a $y\in \mathcal{W}_n'$ for which
	\begin{displaymath}
\phi _\gamma (v(y)) <-\tilde{z}_\gamma ^{(n)}\quad \textup{if}\quad \gamma \in H,
	\end{displaymath}
and
	\begin{displaymath}
\phi _\gamma (v(y)) \ge -\tilde{z}_\gamma ^{(n)}\quad \textup{if}\quad \gamma \in J_n^+\smallsetminus H.
	\end{displaymath}
According to Lemma \ref{l:induction} this means that $\dim V_{F_n^{-1}} \ge |J_n^+| \ge \beta |F_n|$. If we set $Q=F_n^{-1}$, where $n$ is sufficiently large, we obtain a contradiction to Lemma \ref{l:dimension}. This contradiction shows that $\htop(\bar{\lambda }_{\bar{Y}_f}) \le \htop(\lambda _f)$ and completes the proof of Theorem \ref{t:comparison2}.
	\end{proof}

	\begin{lemma}
	\label{l:comparison31}
Any $\bar{\lambda }$-invariant Borel probability measure $\nu $ on $\bar{Y}_f$ satisfies that $h_\nu (\bar{\lambda }_{\bar{Y}_f}) = \linebreak[4]h_{\eta _*\nu }(\lambda _f)$.
	\end{lemma}

	\begin{proof}
For every $x\in X_f$ we denote by $\htop\bigl(\bar{\lambda }_{\bar{Y}_f}| \eta ^{-1}(x)\bigr)$ the fibre entropy of $\bar{\lambda }_{\bar{Y}_f}$, given $x$, defined in \cite{Li}*{Definition 6.7}. The proof of Theorem \ref{t:comparison2} shows that $\htop\bigl(\bar{\lambda }_{\bar{Y}_f}| \eta ^{-1}(x)\bigr)=0$ for every $x\in X_f$. By \cite{Li}*{Lemmas 6.8 and 6.9}, $h_\nu (\bar{\lambda }_{\bar{Y}_f}| \eta ^{-1}(\mathcal{B}_{X_f}))=0$ for every $\bar{\lambda }$-invariant Borel probability measure $\nu $ on $\bar{Y}_f$, where $\mathcal{B}_{X_f}$ is the Borel $\sigma $-algebra of $X_f$. By \cite{Danilenko}*{Theorem 0.2} or \cite{KL16}*{Theorem 9.16}, $h_\nu (\bar{\lambda }_{\bar{Y}_f}) = h_\nu (\bar{\lambda }_{\bar{Y}_f}| \mathcal{B}_{X_f}) + h_{\eta _*\nu }(\lambda _f) = h_{\eta _*\nu }(\lambda _f)$.
	\end{proof}

The coincidence of topological entropies of the $\Gamma $-actions $\lambda _f$ and $\bar{\lambda }_{\bar{Y}_f}$ in Theorem \ref{t:comparison2} is not quite as obvious as one might think. As noted in the proof of Lemma \ref{l:comparison31}, the conditional fibre entropy $\htop (\bar{\lambda }_{\bar{Y}_f}|\linebreak[0]\eta ^{-1}(x))$ is equal to zero for every $x\in X_f$ whenever $\Gamma $ is amenable and $f\in \mathbb{Z}\Gamma $ is not a left zero divisor. This is no longer true if $f$ \textsl{is} a left zero divisor (in which case the topological entropies $\htop(\lambda _f)$ and $\htop(\bar{\lambda }_{\bar{Y}_f})$ are infinite by Lemma \ref{l:zero divisor}). A slight modification of the proof of Lemma \ref{l:zero divisor} yields the following result.

	\begin{proposition}
	\label{L-positive fibre}
Let $\Gamma $ be a countably infinite amenable group, and let $f\in \mathbb{Z}\Gamma $ be a left zero divisor in $\mathbb{R}\Gamma $. Then the fibre entropy $\htop(\bar\lambda _{\bar{Y}_f}|\eta ^{-1}(0_{X_f}))$ is positive.
	\end{proposition}

	\begin{proof}
Take a compatible metric $\mathsf{d}$ on $\bar{Y}_f$ such that $\mathsf{d}(y, z)\ge |y_{1_\Gamma }-z_{1_\Gamma }|$ for all $y, z\in \bar{Y}_f$.

If $f\in \mathbb{Z}\Gamma $ is a left zero divisor we choose $g=\sum_{\gamma \in \Gamma }g_\gamma \cdot \gamma \in \mathbb{R}\Gamma $ with $g_{1_\Gamma }=\|g\|_\infty >0$ and $fg=0$. Following the proof of Lemma \ref{l:zero divisor} we note that $cg^*\in W_f$ and $\eta (cg^*)\in X_f$ for every $c\in \mathbb{R}$. Put $E=\textup{supp}(g^*)$ and choose a maximal set $D\subset \Gamma $ such that the translates $\{\delta E\mid \delta \in D\}$ are disjoint. Then $DEE^{-1}=\Gamma $ (cf. the proof of Lemma \ref{l:zero divisor}). Since the sets $\delta E, \delta \in D$, are disjoint, we obtain, for every $z=(z_\delta )_{\delta \in D}\in \{-1,1\}^D$ and every $c\in \mathbb{R}$ with $0<c<\frac {1}{2\|g\|_\infty }$, a point $w^{(c,z)} = c\cdot \sum_{\delta \in D}z_\delta \bar{\lambda }^\delta g^*\in W_f$ with $\|w^{(c,z)}\|_\infty = c\|g\|_\infty $ and $w^{(c,z)}_\delta =cz_\delta \|g\|_\infty $ for every $\delta \in D$.

We set $x^{(c,z)}=\eta (w^{(c,z)})\in X_f$ and denote by $y^{(c,z)}$ the unique point in $Y_f$ satisfying $\eta (y^{(c,z)})\linebreak[0] = x^{(c,z)}= \eta (w^{(c,z)})$. For every $\delta \in D$,
	\begin{displaymath}
y^{(c,z)}_\delta =
	\begin{cases}
c\|g\|_\infty &\textup{if}\enspace z_\delta =1,
	\\
1-c\|g\|_\infty &\textup{if}\enspace z_\delta = -1.
	\end{cases}
	\end{displaymath}
As $c\searrow 0$, $y^{(c,z)}$ converges coordinate-wise to a point $y^{(z)}\in \bar{Y}_f$ with
	\begin{displaymath}
y^{(z)}_\delta =
	\begin{cases}
0&\textup{if}\enspace z_\delta =1,
	\\
1&\textup{if}\enspace z_\delta =-1,
	\end{cases}
	\end{displaymath}
for $\delta \in D$. With the exception of the single point $z'=(z_\delta ')_{\delta \in D}$ with $z_\delta '=1$ for every $\delta \in D$, all the points $y^{(z)}, z\in \{-1,1\}^D$ lie in $\bar{Y}_f\smallsetminus Y_f$ and satisfy that $\eta (y^{(z)}) = 0_{X_f}$. As in the proof of Lemma \ref{l:zero divisor} we conclude that the fibre entropy $\htop(\bar{\lambda }_{\bar{Y}_f}|\eta ^{-1}(0_{X_f}))$ is positive.
	\end{proof}

	\begin{definition}[\cite{Weiss}]
	\label{d:intrinsic}
A continuous action $\tau $ of a countably infinite amenable group $\Gamma $ on a compact metrizable space $X$ is \emph{intrinsically ergodic} if it has finite topological entropy and there exists a unique $\tau $-invariant Borel probability measure $\mu $ on $X$ with $h_\mu (\tau )=\htop(\tau )$.
	\end{definition}

If $\Gamma $ is a countably infinite amenable group, and if $f\in \mathbb{Z}\Gamma $ satisfies that $\htop(\lambda _f)<\infty $, then the principal algebraic action $\lambda _f$ on $X_f$ is intrinsically ergodic (with unique maximal measure $\mu _f$) if and only if $\lambda _f$ has completely positive entropy w.r.t. $\mu _f$ (\cite{Chung-Li}*{Theorem 8.6}). If $\lambda _f$ is intrinsically ergodic on $X_f$, the next result extends this property to the $\Gamma $-actions $\bar{\lambda }_{\bar{Y}_f}$ and $\bar{\lambda }_{\bar{Z}_f}$.

	\begin{proposition}
	\label{p:comparison3}
Suppose that $\Gamma $ is a countably infinite discrete amenable group, $f\in \mathbb{Z}\Gamma $, and $\lambda _f$ is intrinsically ergodic on $X_f$. Then the following are true.
	\begin{enumerate}
	\item
The $\Gamma $-actions $\bar{\lambda }_{\bar{Y}_f}$ and $\bar{\lambda }_{\bar{Z}_f}$ are intrinsically ergodic;
	\item
The maximal entropy measures of the $\Gamma $-actions $\bar{\lambda }_{\bar{Y}_f}$ and $\bar{\lambda }_{\bar{Z}_f}$ have completely positive entropy.
	\end{enumerate}
	\end{proposition}

The proof of Proposition \ref{p:comparison3} consists of three lemmas.

	\begin{lemma}
	\label{l:comparison33}
If $\lambda _f$ is intrinsically ergodic on $X_f$, then the $\Gamma $-actions $\bar{\lambda }_{\bar{Y}_f}$ on $(\bar{Y}_f,\nu _f)$ and $\bar{\lambda }_{\bar{Z}_f}$ on $(\bar{Z}_f,\nu _f^\hash )$ {\textup(}with $\nu _f^\hash \coloneqq \bar{\rho }^f_*\nu _f${\textup)} have completely positive entropy.
	\end{lemma}

	\begin{proof}
Since $\lambda _f$ on $(X_f,\mu _f)$ is measurably conjugate to $\bar{\lambda }_{\bar{Y}_f}$ on $(\bar{Y}_f,\nu _f)$ by Lemma \ref{l:comparison32}, and $\bar{\lambda }_{\bar{Z}_f}$ on $(\bar{Z}_f,\nu _f^\hash)$ is a factor of $\bar{\lambda }_{\bar{Y}_f}$ on $(\bar{Y}_f,\nu _f)$, all these actions have completely positive entropy.
	\end{proof}

	\begin{proof}[Proof of Proposition \ref{p:comparison3}]
If $\nu $ is a $\bar{\lambda }$-invariant Borel probability measure on $\bar{Y}_f$ with entropy $h_\nu (\bar{\lambda }_{\bar{Y}_f})=\htop(\bar{\lambda }_{\bar{Y}_f})=\htop(\lambda _f)$ (cf. Theorem \ref{t:comparison2}), then Lemma \ref{l:comparison31} implies that $\eta _*\nu =\mu _f$, the unique $\lambda _f$-invariant Borel probability measure on $X_f$ with maximal entropy. By Lemma \ref{l:comparison32}, $\nu =\nu _f$, and the $\Gamma $-actions $\lambda _f$ and $\bar{\lambda }_{\bar{Y}_f}$ on $(X_f,\mu _f)$ and $(\bar{Y}_f,\nu _f)$ are conjugate. Lemma \ref{l:comparison33} completes the proof of Proposition \ref{p:comparison3}.
	\end{proof}

	\begin{theorem}
	\label{t:symbolic}
Suppose that $\Gamma $ is a countably infinite amenable group, $f\in \mathbb{Z}\Gamma $, and the principal algebraic $\Gamma $-action $\lambda _f$ on $X_f$ is intrinsically ergodic. Then the principal algebraic $\Gamma $-action $\lambda _f$ on $(X_f,\mu _f)$ is measurably conjugate to the $\Gamma $-actions $\bar{\lambda }_{\bar{Y}_f}$ and $\bar{\lambda }_{\bar{Z}_f}$ on $(\bar{Y}_f, \nu _f)$ and $(\bar{Z}_f, \nu _f^\hash)$, respectively.
	\end{theorem}

	\begin{proof}
If $\lambda _f$ is intrinsically ergodic on $X_f$, then $\htop(\lambda _f)<\infty $ and $\mu _f$ has c.p.e. (cf. Definition \ref{d:intrinsic}). Lemma \ref{l:comparison32} shows that the $\Gamma $-actions $\lambda _f$ on $(X_f,\mu _f)$ and $\bar{\lambda }$ on $(\bar{Y}_f, \nu _f)$ are measurably conjugate, and the $\Gamma $-actions $\bar{\lambda }_{\bar{Y}_f}$ and $\bar{\lambda }_{\bar{Z}_f}$ on $(\bar{Y}_f, \nu _f)$ and $(\bar{Z}_f, \nu _f^\hash)$ are measurably conjugate by Corollary \ref{c:conjugacy}.
	\end{proof}

	\section{Generators of intrinsically ergodic principal algebraic actions}
	\label{s:generators}

In this section we apply Theorem \ref{t:symbolic} to find generators of intrinsically ergodic principal algebraic actions of a countably infinite amenable group $\Gamma $.

Let $f\in \mathbb{Z}\Gamma $ be a nonzero element such that the principal algebraic $\Gamma $-action $\lambda _f$ on $X_f$ is intrinsically ergodic. We view $X_f\subset \mathbb{T}^\Gamma $ as a subset of $[0,1)^\Gamma $ as in \eqref{eq:eta} -- \eqref{eq:Wf} by identifying $Y_f$ and $X_f$ through $\eta $ and set, for every $j\in \mathbb{Z}$,
	\begin{equation}
	\label{eq:B}
B[j] = \bigl\{x\in X_f\,\mid\,\textstyle{\sum_{\gamma \in \textup{supp}(f)}} f_\gamma x_\gamma = j \bigr\} = \bigl\{x\in X_f\mid (\bar{\rho }^fx)_{1_\Gamma }=j\bigr\}.
	\end{equation}

The following corollaries are immediate consequences of Theorem \ref{t:symbolic}.

	\begin{corollary}
	\label{c:B}
Put
	\begin{equation}
	\label{eq:Bfpm}
\mathcal{B}_f =
	\begin{cases}
\{B[j]\mid j = -\|f^-\|_1 +1, \dots, \|f^+\|_1 - 1\}&\textup{if both $f^+$ and $f^-$ are nonzero},
	\\
\{B[j]\mid j = 0,\dots, \|f^+\|_1 - 1\}&\textup{if $f^+\ne 0$ and $f^-=0$},
	\\
\{B[j]\mid j = -\|f^-\|_1 +1 ,\dots ,0\}&\textup{if $f^+= 0$ and $f^-\ne0$}.
	\end{cases}
	\end{equation}
Then $\mathcal{B}_f$ is a Borel partition of $X_f$ which is a generator $(\textup{mod}\;\mu _f)$ for $\lambda _f$.
	\end{corollary}

	\begin{corollary}
	\label{c:C}
The Borel partition $\mathcal{C}_f = \{C_j\mid j=0,\dots ,\|f\|_1-1\}$ of $X_f$ with
	\begin{displaymath}
C_j = \bigl\{x\in X_f\mid j/\|f\|_1 \le x_{1_\Gamma } < (j+1)/\|f\|_1\;(\textup{mod}\;1)\}\enspace \textup{for}\enspace j=0,\dots ,\|f\|_1-1,
	\end{displaymath}
is a generator $(\textup{mod}\;\mu _f)$ for $\lambda _f$.
	\end{corollary}

By imposing further conditions on $\Gamma $ and $f$ we can sometimes find slightly smaller generators $(\textup{mod}\;\mu _f)$ for $\lambda _f$ in Corollary \ref{c:B}.

	\begin{corollary}
	\label{c:B'}
Suppose that the group $\Gamma $ in Theorem \ref{t:symbolic} is left {\textup(}or, equivalently, right{\textup)} orderable. If $f\in \mathbb{Z}\Gamma $ satisfies that $|\textup{supp}(f)|\ge2$, then the collection of sets
	\begin{displaymath}
\mathcal{B}_f' = \bigl\{B[j]\mid j = -\|f^-\|_1 +1, \dots, \|f^+\|_1 - 1\bigr\},
	\end{displaymath}
defined as in \eqref{eq:B}, is a generator $(\textup{mod}\;\mu _f)$ for $\lambda _f$.
	\end{corollary}

For the proof of Corollary \ref{c:B'} we require an additional lemma. For notation we refer to Lemma \ref{l:comparison32}.

	\begin{lemma}
	\label{l:K}
Suppose that the group $\Gamma $ in Theorem \ref{t:symbolic} is left {\textup(}or, equivalently, right{\textup)} orderable. If $f\in \mathbb{Z}\Gamma $ satisfies that $|\textup{supp}(f)|\ge2$, then $K = \pi _{1_\Gamma }(X_f) = \mathbb{T}$ and hence $(\pi _{1_\Gamma })_*\mu _f=\mu _\mathbb{T}$, the Lebesgue measure on $\mathbb{T}$.
	\end{lemma}

	\begin{proof}
Since $K\subset \mathbb{T}$ is a closed subgroup, it is either finite or equal to $\mathbb{T}$. If $K$ is finite we choose $L\ge1$ such that $LK=\{Lt\mid t\in K\} = \{0\}$ and conclude from \eqref{eq:Xf} that $L$ lies in the left ideal $(f)\subset \mathbb{Z}\Gamma $ generated by $f$. Then there exists $h\in \mathbb{Z}\Gamma $ such that $L=hf$ or, equivalently, $1=h\cdot \frac1L f$. In other words, the rational group ring $\mathbb{Q}\Gamma $ contains the nontrivial unit $\frac1L f$, in violation of \cite{Passman}*{Lemmas 13.1.7 and 13.1.10}.

If $\pi _{1_\Gamma }(X_f) = \mathbb{T}$, then obviously $(\pi _{1_\Gamma })_*\mu _f=\mu _\mathbb{T}$.
	\end{proof}

	\begin{proof}[Proof of Corollary \ref{c:B'}]
If both $f^+$ and $f^-$ are nonzero our assertion follows from Corollary \ref{c:B}. If $f^-=0$, then
	\begin{displaymath}
B[0] = \{x\in X_f\mid \textstyle\sum_{\gamma \in \textup{supp}(f)}f_\gamma x_\gamma =0\} = \{x\in X_f\mid x_\gamma =0\enspace \textup{for every} \enspace \gamma \in \textup{supp}(f)\}.
	\end{displaymath}
Hence $\mu _f(B[0])\le \mu _f(\{x\in X_f\mid x_\gamma =0\})$ for every $\gamma \in \Gamma $, so that $\mu _f(B[0]) = 0$ by Lemma \ref{l:K}. By Corollary \ref{c:B}, $\{B[j]\mid j = 1, \dots ,\|f^+\|_1 - 1\}$ is a generator $(\textup{mod}\;\mu _f)$ for $\lambda _f$.

If $f^+=0$ the proof is completely analogous.
	\end{proof}

\section{Examples}
	\label{s:examples}

Let $\Gamma $ be a countably infinite discrete amenable group, $f\in \mathbb{Z}\Gamma $, and let $\lambda _f$ be the principal algebraic $\Gamma $-action on $X_f$ in Definition \ref{d:principal}. In order to apply Theorem \ref{t:symbolic} and its corollaries to $\lambda _f$ we require the action $\lambda _f$ to be intrinsically ergodic.

\subsection{Intrinsically ergodic principal algebraic $\mathbb{Z}^d$-actions}
	\label{ss:Z-d}
If $\Gamma =\mathbb{Z}^d$ for some $d\ge1$, the conditions for principal algebraic $\Gamma $-actions to be intrinsically ergodic are well understood: if $f$ is nonzero and not divisible by a generalized cyclotomic polynomial, then $\lambda _f$ is intrinsically ergodic \cite{DSAO}*{Theorem 11.2, Propositions 19.4 and 20.5}.

	\begin{example}
	\label{e:4-torus}
The matrix $M=\left[
	\begin{smallmatrix}
\hphantom{-}0&1&0&0
	\\
\hphantom{-}0&0&1&0
	\\
\hphantom{-}0&0&0&1
	\\
-1&1&1&1
	\end{smallmatrix}
\right] \in \textup{SL}(4,\mathbb{Z})$ defines a nonhyperbolic ergodic automorphism $\alpha _M$ of $\mathbb{T}^4$. The question of finding `nice' finite generating partitions for such automorphisms was discussed in \cite{Lind}*{Theorem 1}. By observing that $\alpha _M$ is algebraically conjugate to the intrinsically ergodic algebraic $\mathbb{Z}$-action $\lambda _f$ on $(X_f,\mu _f)$ for the characteristic polynomial $f=u^4-u^3-u^2-u+1$ of $M$ and applying Corollary \ref{c:C} we see that the sets
	\begin{displaymath}
C_j = \bigl\{x\in X_f\mid \tfrac{j}{5}\le x_0 < \tfrac{j+1}{5} \enspace (\textup{mod}\;1)\},\; 0\le j\le 4,
	\end{displaymath}
form a generating partition for $\lambda _f$ w.r.t. $\mu _f$ on $X_f$. When translating this information back to $\alpha _M$ we obtain the generator
	\begin{displaymath}
\mathcal{D} = \bigl\{D_j = \{t=(t_1,t_2,t_3,t_4)\in \mathbb{T}^4\mid \tfrac{j}{5} \le t_1 < \tfrac{j+1}{5}\enspace (\textup{mod}\;1)\}\bigm|j=0,\dots,4\bigr\}
	\end{displaymath}
for $\alpha _M$ w.r.t. Lebesgue measure $\mu _{\mathbb{T}^4}$ on $\mathbb{T}^4$.

Corollary \ref{c:B} shows that $\alpha _M$ on $(\mathbb{T}^4,\mu _{\mathbb{T}^4})$ also has the $4$-element generator corresponding to $\{\negthinspace B[-2], B[-1],B[0],B[1]\}$ in \eqref{eq:Bfpm}.
	\end{example}

	\begin{example}
	\label{e:1+x+y}
Let $\Gamma =\mathbb{Z}^2$, and let $f=1-u_1-u_2\in \mathbb{Z}[u_1^{\pm1},u_2^{\pm1}]\cong \mathbb{Z}[\mathbb{Z}^2]$. By \cite{DSAO}*{Proposition 19.7}, $\htop (\lambda _f)= \frac{3\sqrt{3}}{4\pi }L(2,\chi _3) > 0$, where $L(2,\chi _3)$ is the Dirichlet $L$-function defined there. Since $\lambda _f$ is intrinsically ergodic, Corollary \ref{c:B} shows that $\lambda _f$ on $(X_f,\mu _f)$ has the 2-element generator $\{\negthinspace B[-1],B[0]\}$ $(\textup{mod}\;\mu _f)$ defined as in \eqref{eq:B}.
	\end{example}

\subsection{Intrinsically ergodic principal algebraic actions of amenable groups}
	\label{ss:homoclinic}
If $\Gamma $ is a countably infinite discrete amenable group, establishing the intrinsic ergodicity of a principal algebraic action $\lambda _f$, $f\in \mathbb{Z}\Gamma $, is much more delicate than for $\Gamma =\mathbb{Z}^d$. A sufficient condition for intrinsic ergodicity of $\lambda _f$ can be expressed in terms of homoclinic points: a point $x = (x_s)\in X_f$ is \emph{summable homoclinic} if $\sum_{s \in \Gamma }|\negthinspace|x_s|\negthinspace| < \infty $, where $|\negthinspace|t|\negthinspace|$ denotes the distance from $0$ of a point $t\in \mathbb{T}$ (cf. \cite{LSV2}). Clearly, every summable homoclinic point $x\in X_f$ is \emph{homoclinic}, i.e., $\lim_{s \to \infty }\lambda _f^sx = 0$ (cf. e.g., \cite{homoclinic}*{Definition 3.1} or \cite{Chung-Li}).

	\begin{proposition}
	\label{p:homoclinic}
Let $\Gamma $ be a countably infinite amenable discrete group, $f\in \mathbb{Z}\Gamma $, and let $\Delta ^1(X_f)\linebreak[0]\subset X_f$ be the group of summable homoclinic points of the principal algebraic $\Gamma $-action $\lambda _f$ on $X_f$. If $\Delta ^1(X_f)$ is dense in $X_f$ and $\htop(\lambda _f)<\infty $ then $\lambda _f$ is intrinsically ergodic.
	\end{proposition}

The converse of Proposition \ref{p:homoclinic} is clearly not true: the principal algebraic $\mathbb{Z}$-action $\lambda _f$ on $X_f$ (or, equivalently, the automorphism $\alpha _M$ of $\mathbb{T}^4$) in Example \ref{e:4-torus} is intrinsically ergodic, but has no nonzero homoclinic points (cf. \cite{homoclinic}*{Example 3.4}).

	\begin{proof}[Proof of Proposition \ref{p:homoclinic}]
According to \eqref{eq:Xf}, the dual group $\widehat{X_f}$ is given by $\widehat{X_f}=
\mathbb{Z}\Gamma /(f)$ and is, in particular, a finitely generated left $\mathbb{Z}\Gamma $-module. By \cite{Chung-Li}*{Theorem 7.8}, $\Delta ^1(X_f) \subset \textup{IE}(X_f)$, the closed subgroup of $X_f$ defined in \cite{Chung-Li}*{Definition 7.2}, and hence $\textup{IE}(X_f) = X_f$ by assumption. By \cite{Chung-Li}*{Corollary 8.4 and Theorem 8.6}, $\lambda _f$ is intrinsically ergodic.
	\end{proof}

\subsubsection{Expansive principal algebraic actions}
	\label{sss:expansive}

For every countably infinite discrete group $\Gamma $ and every $f\in \mathbb{Z}\Gamma $, the principal algebraic $\Gamma $-action $\lambda _f$ on $X_f$ in Definition \ref{d:principal} is expansive if and only if the map $\bar{\rho }^f\colon \ell ^\infty (\Gamma ,\mathbb{R}) \to \ell ^\infty (\Gamma ,\mathbb{R})$ in \eqref{eq:convolutions} is injective or, equivalently, if $f$ is invertible in $\ell ^1(\Gamma ,\mathbb{R})$ (\cite{Deninger-S}*{Theorem 3.2}). If this is the case, $w^\Delta \coloneqq (f^*)^{-1} \in W_f$ since $\bar{\rho }^f(w^\Delta )=1_\Gamma $. By \cite{Deninger-S}*{Proposition 4.2}, the map $\bar{\rho }^{w^\Delta }\colon \ell ^\infty (\Gamma ,\mathbb{Z}) \to \ell ^\infty (\Gamma ,\mathbb{R})$ is continuous in the weak$^*$-topology on closed, bounded subsets of $\ell ^\infty (\Gamma ,\mathbb{Z})$, and the map $\xi \coloneqq \eta \circ \bar{\rho }^{w^\Delta } \colon \ell ^\infty (\Gamma ,\mathbb{Z})\to \mathbb{T}^\Gamma $ satisfies that $\xi (\{v\in \ell ^\infty (\Gamma ,\mathbb{Z})\mid \|v\|_\infty \le \|f\|_1/2\}) = X_f$ (cf. \cite{Deninger-S}*{Lemma 4.5}). Since $\xi (\mathbb{Z}\Gamma )\subset \Delta ^1(X_f)$, the continuity of $\xi $ implies that $\Delta ^1(X_f)$ is dense in $X_f$. If $\Gamma $ is amenable, we conclude from Proposition \ref{p:homoclinic} that every expansive principal algebraic $\Gamma $-action is intrinsically ergodic.

If $\Gamma $ is amenable and $\lambda _f$ is expansive, the partitions $\mathcal{B}_f$ and $\mathcal{C}_f$ in the Corollaries \ref{c:B} and \ref{c:C} are obviously generators (and not only generators $(\textup{mod}\;\mu _f)$) for $\lambda _f$.

\medskip We mention two examples, taken from \cite{E-R}. Let $\mathbb{H}\subset \textup{SL}(3,\mathbb{Z})$ be the discrete Heisenberg group, generated by
	\begin{equation}
	\label{H-generators}\smash[b]{
u_1 = \left[
	\begin{smallmatrix}
1&0&0
	\\
0&1&1
	\\
0&0&1
	\end{smallmatrix}
\right],
\quad
u_2 = \left[
	\begin{smallmatrix}
1&1&0
	\\
0&1&0
	\\
0&0&1
	\end{smallmatrix}
\right],
\quad \text{and}\quad
u_3 = \left[
	\begin{smallmatrix}
1&0&1
	\\
0&1&0
	\\
0&0&1
	\end{smallmatrix}
\right].}
	\end{equation}

	\begin{example}[\cite{E-R}*{Example 8.4}]
	\label{e:E-R-1}
Let $f = |a_1|+|a_2|+|a_3|+a_1\cdot u_1+a_2\cdot u_2 + a_3\cdot u_3 \in \mathbb{Z}\mathbb{H}$. Then the principal algebraic $\mathbb{H}$-action $\lambda _f$ is expansive if and only if $a_1a_2\ne0$ and $a_3>0$. In these cases the partitions $\mathcal{B}_f$ and $\mathcal{C}_f$ in the Corollaries \ref{c:B} -- \ref{c:C} are generators for $\lambda _f$.
	\end{example}

	\begin{example}
	\label{e:1+x+y-H}
Let $f= 1-u_1-u_2 \in \mathbb{Z}\mathbb{H}$. The principal algebraic $\mathbb{H}$-action $\lambda _f$ on $(X_f,\mu _f)$ has zero entropy by \cite{Deninger}*{Theorem 11} or \cite{Lind-S}*{Theorem 9.2}, so that the hypotheses of Theorem \ref{t:symbolic} are not satisfied. We do not know whether the $\mathbb{H}$-actions $\lambda _f$ on $(X_f,\mu _f)$ and $\bar{\lambda }_{\bar{Z}_f}$ on $(\bar{Z}_f,\nu ^\hash_f)$ are measurably conjugate and whether $\mathcal{B}_f = \{\negthinspace B[-1],B[0]\}$ is a generator $(\textup{mod}\;\mu _f)$ for $\lambda _f$.
	\end{example}

\section{Summable homoclinic points of nonexpansive principal algebraic actions}

The existence of a nonzero summable homoclinic point for a nonexpansive principal algebraic action in the second part of Example \ref{e:1+x+y-H} is an interesting fact. For $\Gamma =\mathbb{Z}^d$ this phenomenon is well understood: for $f = \sum_{\mathbf{n}\in \mathbb{Z}^d}f_\mathbf{n}u^\mathbf{n}\in \mathbb{Z}[\mathbb{Z}^d]$ we denote by
	\begin{equation}
	\label{eq:variety}
\mathsf{U}(f) = \{\mathbf{z}=(z_1,\dots ,z_d)\in \mathbb{S}^d\mid f(\mathbf{z})=0\}
	\end{equation}
the \emph{unitary variety} of $f$, where $\mathbb{S}=\{z\in \mathbb{C}\mid |z|=1\}$. According to \cite{DSAO}*{Theorem 6.5}, $\lambda _f$ is expansive if and only if $\mathsf{U}(f)=\varnothing $. If $f$ is nonzero and irreducible, then \cite{LSV2}*{Theorem 3.2} shows that $\Delta ^1(X_f) \ne \{0\}$ if and only if the dimension of $\mathsf{U}(f)\subset \mathbb{S}^d$ is $\le d-2$. In this case $\Delta ^1(X_f)$ is dense in $X_f$, $\lambda _f$ is intrinsically ergodic, and the partitions $\mathcal{B}_f$ and $\mathcal{C}_f$ are generators $(\textup{mod}\;\mu _f)$ for $\lambda _f$.

For nonabelian groups $\Gamma $, examples of nonexpansive principal algebraic actions with summable homoclinic points are much harder to come by. In order to present a class of such actions we assume for the remainder of this section that $\Gamma $ is a countably infinite discrete group with center $H$. We say $f\in \mathbb{R} \Gamma$ is \emph{well-balanced} (\cite{Bowen-Li}*{Definition 1.2}) if
	\begin{enumerate}
	\item
$\sum_{s\in \Gamma}f_s=0$,
	\item
$f_s\le 0$ for every $s\in \Gamma\smallsetminus \{1_\Gamma\}$,
	\item
$f=f^*$,
	\item
$\textup{supp}(f)$ generates $\Gamma$.
	\end{enumerate}

We shall prove the following theorem.

	\begin{theorem}
	\label{T-Harmonic}
Assume that for any finite $F\subseteq \Gamma$ there is some $s$ in the center $H$ of $\Gamma$ such that none of $s, s^2, s^3$ is in $F$ {\textup(}this happens for example when $H^6=\{s^6\mid s\in H\}$ is infinite{\textup)}. Also assume that $\Gamma$ is not virtually $\mathbb{Z}$ or $\mathbb{Z}^2$. Let $f\in \mathbb{Z}\Gamma$ be well-balanced. Then $\Delta ^1(X_f)$ is dense in $X_f$.
	\end{theorem}

	\begin{corollary}
	\label{c:balanced}
Assume that the group $\Gamma $ in Theorem \ref{T-Harmonic} is amenable. If $f\in \mathbb{Z}\Gamma $ is well-balanced, then the principal algebraic $\Gamma $-action $\lambda _f$ on $X_f$ is intrinsically ergodic.
	\end{corollary}

	\begin{proof}
The proof of Corollary \ref{C-summable} implies that $f$ is not a left zero divisor in $\mathbb{Z}\Gamma $, so that $\htop(\lambda _f)<\infty $ by Lemma \ref{l:zero divisor}. Now apply Theorem \ref{T-Harmonic} and Proposition \ref{p:homoclinic}.
	\end{proof}

The proof of Theorem \ref{T-Harmonic} requires a brief excursion into Banach algebras.

	\begin{lemma}
	\label{L-summable}
Let $\mathcal{A}$ be a unital Banach algebra such that $\|ab\|\le \|a\|\cdot \|b\|$ for all $a, b\in \mathcal{A}$.
Let $0<c<1$ and $q, r\in \mathcal{A}$ such that $\|r\|\le 1-c$ and $\|q\|\le c$.
Then
	\begin{displaymath}
\sum_{k=0}^\infty \|r^k(c-q)^3(1-q)^{-(k+1)}\|<\infty.
	\end{displaymath}
	\end{lemma}

The proof of Lemma \ref{L-summable} requires some auxiliary results.

	\begin{lemma}
	\label{L-bound for combinatorial}
Let $0<c<1$. For any $x, y>0$ one has
	\begin{displaymath}
\frac{c^x (1-c)^y (x+y)^{x+y}}{x^x y^y}\le 1.
	\end{displaymath}
	\end{lemma}

	\begin{proof}
Fix $y>0$ and put
	\begin{displaymath}
\phi (x)=\log \frac{c^x (1-c)^y (x+y)^{x+y}}{x^x y^y}=x\log c+y\log (1-c)+(x+y)\log (x+y)-x\log x-y\log y
	\end{displaymath}
for $x>0$. Then
	\begin{displaymath}
\phi '(x)=\log c+\log (x+y)-\log x=\log \frac{c(x+y)}{x}.
	\end{displaymath}
Clearly $\phi '>0$ on the interval $(0, \frac{cy}{1-c})$, $\phi '=0$ at $\frac{cy}{1-c}$, and $\phi '<0$ on $(\frac{cy}{1-c}, \infty)$. Thus $\phi $ attains its maximum value at $\frac{cy}{1-c}$. Since $\phi (\frac{cy}{1-c})=0$, one concludes that $\phi (x)\le 0$ for all $x>0$.
	\end{proof}

	\begin{lemma}
	\label{L-polynomial}
Let $0<c<1$ and $k\in \mathbb{N}$. Let $g_k$ be the cubic polynomial given by
	\begin{align}
	\label{E-poly}
g_k(x)& =-(x+1)(x+2)(x+3)+3c(x+k+1)(x+2)(x+3)
	\\
\nonumber &\quad -3c^2(x+k+1)(x+k+2)(x+3)+c^3(x+k+1)(x+k+2)(x+k+3)
	\\
\nonumber &=(c-1)^3x^3+(3c(c-1)^2k+6(c-1)^3)x^2
	\\
\nonumber &\quad +(3c^2(c-1)k^2+3c(c-1)(4c-5)k+11(c-1)^3)x
	\\
\nonumber &\quad +c^3k^3+(6c^3-9c^2)k^2+(11c^3-27c^2+18c)k+6(c-1)^3.
	\end{align}
For $\eta>0$, put
	\begin{align}
	\label{E-vary}
y_{k, \eta, \pm}=\frac{ck}{1-c}-2\pm \frac{1}{1-c}\sqrt{\eta k+\tfrac{1}{3}(1-c)^2}.
	\end{align}
Then
	\begin{align*}
g_k(y_{k, \eta, \pm})=(c^2+c)k\pm \bigl((3c-\eta)k+\tfrac{2}{3}(c-1)^2\bigr)\sqrt{\eta k+\tfrac{1}{3}(1-c)^2}.
	\end{align*}
	\end{lemma}
	\begin{proof}
This is a direct computation:
	\begin{align*}
g_k(y_{k, \eta, \pm})&=(c-1)^3y_{k, \eta, \pm}^3+(3c(c-1)^2k+6(c-1)^3)y_{k, \eta, \pm}^2
	\\
&\quad +(3c^2(c-1)k^2+3c(c-1)(4c-5)k+11(c-1)^3)y_{k, \eta, \pm}
	\\
&\quad +c^3k^3+(6c^3-9c^2)k^2+(11c^3-27c^2+18c)k+6(c-1)^3
	\\
&= (c-1)^3(\tfrac{ck}{1-c}-2)((\tfrac{ck}{1-c}-2)^2+3\tfrac{1}{(1-c)^2}(\eta k+\tfrac{1}{3}(1-c)^2))
	\\
&\quad +(3c(c-1)^2k+6(c-1)^3)((\tfrac{ck}{1-c}-2)^2+\tfrac{1}{(1-c)^2}(\eta k+\tfrac{1}{3}(1-c)^2))
	\\
&\quad +(3c^2(c-1)k^2+3c(c-1)(4c-5)k+11(c-1)^3)(\tfrac{ck}{1-c}-2)
	\\
&\quad +c^3k^3+(6c^3-9c^2)k^2+(11c^3-27c^2+18c)k+6(c-1)^3
	\\
&\quad \pm (c-1)^3(\tfrac{ck}{1-c}-2)^2 \tfrac{3}{1-c}\sqrt{\eta k+\tfrac{1}{3}(1-c)^2}
	\\
&\quad \pm (c-1)^3\tfrac{1}{(1-c)^3}(\eta k+\tfrac{1}{3}(1-c)^2)\sqrt{\eta k+\tfrac{1}{3}(1-c)^2}
	\\
&\quad \pm 2(3c(c-1)^2k+6(c-1)^3)(\tfrac{ck}{1-c}-2)\tfrac{1}{1-c}\sqrt{\eta k+\tfrac{1}{3}(1-c)^2}
	\\
&\quad \pm (3c^2(c-1)k^2+3c(c-1)(4c-5)k+11(c-1)^3)\tfrac{1}{1-c}\sqrt{\eta k+\tfrac{1}{3}(1-c)^2}
	\\
&= -(ck+2(c-1))((ck+2(c-1))^2+(3\eta k+(1-c)^2))
	\\
&\quad +(ck+2(c-1))(3(ck+2(c-1))^2+(3\eta k+(1-c)^2))
	\\
&\quad -(3c^2k^2+3c(4c-5)k+11(c-1)^2)(ck+2(c-1))
	\\
&\quad +c^3k^3+(6c^3-9c^2)k^2+(11c^3-27c^2+18c)k+6(c-1)^3
	\\
&\quad \mp 3(ck+2(c-1))^2\sqrt{\eta k+\tfrac{1}{3}(1-c)^2}
	\\
&\quad \mp (\eta k+\tfrac{1}{3}(1-c)^2)\sqrt{\eta k+\tfrac{1}{3}(1-c)^2}
	\\
&\quad \pm 6(ck+2(c-1))^2\sqrt{\eta k+\tfrac{1}{3}(1-c)^2}
	\\
&\quad \mp (3c^2k^2+3c(4c-5)k+11(c-1)^2)\sqrt{\eta k+\tfrac{1}{3}(1-c)^2}
	\\
&= (c^2+c)k\pm ((3c-\eta)k+\tfrac{2}{3}(c-1)^2)\sqrt{\eta k+\tfrac{1}{3}(1-c)^2}.\tag*{\qedsymbol}
	\end{align*}
\renewcommand{\qedsymbol}{}\vspace{-\baselineskip}
	\end{proof}

	\begin{lemma}
	\label{L-root}
Fix $0<c<1$. Then there is some $k_c\in \mathbb{N}$ such that for each $k\in \mathbb{N}$ with $k\ge k_c$ the following hold:
	\begin{enumerate}
	\item
the polynomial $g_k$ given by \eqref{E-poly} has $3$ roots
$t_{k, 1}<t_{k, 2}<t_{k, 3}$ such that
	\begin{displaymath}
1<y_{k, 4c, -}<t_{k, 1}<y_{k, c, -}<t_{k, 2}<y_{k, c, +}<t_{k, 3}<y_{k, 4c, +},
	\end{displaymath}
where $y_{k, \eta, \pm}$ is given in \eqref{E-vary};
	\item
$g_k>0$ on $(-\infty, t_{k, 1})\cup (t_{k, 2}, t_{k, 3})$ and $g_k<0$ on $(t_{k, 1}, t_{k, 2})\cup (t_{k, 3}, \infty)$.
	\end{enumerate}
	\end{lemma}
	\begin{proof}
Take $k_c\in \mathbb{N}$ such that for each $k\ge k_c$ one has
	\begin{displaymath}
\tfrac{ck}{1-c}-2-\tfrac{1}{1-c}\sqrt{4ck+\tfrac{1}{3}(1-c)^2}>1,
	\end{displaymath}
and
	\begin{displaymath}
(c^2+c)k-(-ck+\tfrac{2}{3}(c-1)^2)\sqrt{4c k+\tfrac{1}{3}(1-c)^2}>0,
	\end{displaymath}
and
	\begin{displaymath}
(c^2+c)k-(2ck+\tfrac{2}{3}(c-1)^2)\sqrt{c k+\tfrac{1}{3}(1-c)^2}<0,
	\end{displaymath}
and
	\begin{displaymath}
(c^2+c)k+ (-ck+\tfrac{2}{3}(c-1)^2)\sqrt{4c k+\tfrac{1}{3}(1-c)^2}<0.
	\end{displaymath}
Then for each $k\in \mathbb{N}$ with $k\ge k_c$, one has $y_{k, 4c, -}>1$ and by Lemma~\ref{L-polynomial} one has $g_k(y_{k, 4c, -})>0$, $g_k(y_{k, c, -})<0$, $g_k(y_{k, c, +})>0$, and
$g_k(y_{k, 4c, +})<0$. Since $y_{k, 4c, -}<y_{k, c, -}<y_{k, c, +}<y_{k, 4c, +}$, it follows that (1) holds. As $g_k$ is a cubic polynomial, (2) must also hold.
	\end{proof}

	\begin{lemma}
	\label{L-asymp}
Fix $0<c<1$.
For $k\in \mathbb{N}$ set
	\begin{align}
	\label{E-asymp1}
f_k(m)=c^{m}\tbinom{m+k}{k}-2c^{m+1}\tbinom{m+k+1}{k}+c^{m+2}\tbinom{m+k+2}{k}
	\end{align}
and
	\begin{align}
	\label{E-asymp}
h_k(m)=(1-c)^kf_k(m)
	\end{align}
for $m\in \mathbb{Z}_{\ge 0}$.
Let $\xi: \mathbb{N}\rightarrow \mathbb{N}$ such that $\xi(k)=\frac{ck}{1-c}+\mathcal{O}(k^{1/2})$ as $k\to \infty$. Then
	\begin{align*}
h_k(\xi(k))=\mathcal{O}(k^{-3/2})
	\end{align*}
as $k\to \infty$.
	\end{lemma}

	\begin{proof}
For $k\in \mathbb{N}$ define $\varphi_k, \psi_k: \mathbb{N}\rightarrow \mathbb{R}$ by
	\begin{displaymath}
\varphi_k(m)=(1-c)^kc^m\tbinom{m+k}{k}
	\end{displaymath}
and
	\begin{displaymath}
\psi_k(m)=\tfrac{(m+1)(m+2)-2c(m+k+1)(m+2)+c^2(m+k+1)(m+k+2)}{(m+1)(m+2)}.
	\end{displaymath}

For each $m\in \mathbb{N}$, one has
	\begin{displaymath}
f_k(m)=c^m\tbinom{m+k}{k}\bigl(1-2c\tfrac{m+k+1}{m+1}+ c^2\tfrac{(m+k+1)(m+k+2)}{(m+1)(m+2)}\bigr) =c^m\tbinom{m+k}{k}\psi_k(m),
	\end{displaymath}
whence
	\begin{displaymath}
h_k(m)=\varphi_k(m)\psi_k(m).
	\end{displaymath}
Therefore it suffices to show that $\varphi_k(\xi(k))=\mathcal{O}(k^{-1/2})$ and $\psi_k(\xi(k))=\mathcal{O}(k^{-1})$.

By Stirling's approximation formula there are constants $C_1, C_2>0$ such that
	\begin{displaymath}
C_1\sqrt{2\pi n}(\tfrac{n}{e})^n\le n!\le C_2\sqrt{2\pi n}(\tfrac{n}{e})^n
	\end{displaymath}
for all $n\in \mathbb{N}$. Then
	\begin{displaymath}
\varphi_k(m)=(1-c)^kc^{m}\tfrac{(m+k)!}{m!k!}\le (1-c)^kc^{m}\tfrac{C_2\sqrt{2\pi (m+k)}(m+k)^{m+k}}{C_1^2 2\pi\sqrt{mk}m^{m}k^k} \le \tfrac{C_2\sqrt{2\pi (m+k)}}{C_1^2 2\pi\sqrt{mk}},
	\end{displaymath}
where the 2nd inequality comes from taking $x=m$ and $y=k$ in Lemma~\ref{L-bound for combinatorial}.
Thus
	\begin{displaymath}
\varphi_k(\xi(k))\le \tfrac{C_2\sqrt{2\pi (\xi(k)+k)}}{C_1^2 2\pi\sqrt{\xi(k) k}}=\mathcal{O}(k^{-1/2}),
	\end{displaymath}
whence $\varphi_k(\xi(k))=\mathcal{O}(k^{-1/2})$.

\smallskip Write $\xi(k)$ as $\frac{ck}{1-c}+\lambda_k k^{1/2}$ with $\lambda_k=\mathcal{O}(1)$ as $k\to \infty$. Then
	\begin{align*}
(\xi(k)+1)(\xi(k)+2)&=\tfrac{c^2k^2}{(1-c)^2}+ \tfrac{2c\lambda_k}{1-c}k^{3/2}+\mathcal{O}(k),
	\\
2c(\xi(k)+k+1)(\xi(k)+2)&=\tfrac{2c^2k^2}{(1-c)^2} +\tfrac{2c(1+c)\lambda_k}{1-c}k^{3/2}+\mathcal{O}(k),
	\\
c^2(\xi(k)+k+1)(\xi(k)+k+2)&=\tfrac{c^2k^2}{(1-c)^2}+\tfrac{2c^2\lambda_k}{1-c}k^{3/2}+\mathcal{O}(k),
	\end{align*}
whence
	\begin{displaymath}
(\xi(k)+1)(\xi(k)+2)-2c(\xi(k)+k+1)(\xi(k)+2)+c^2(\xi(k)+k+1)(\xi(k)+k+2)=\mathcal{O}(k).
	\end{displaymath}
It follows that $\psi_k(\xi(k))=\mathcal{O}(k^{-1})$.
	\end{proof}

\medskip For a power series $\phi(x)=\sum_{m=0}^\infty a_m x^m\in \mathbb{C}[[x]]$, we set $|\phi|$ to be the power series
	\begin{displaymath}
|\phi|(x)=\sum_{m=0}^\infty |a_m|x^m.
	\end{displaymath}

	\begin{lemma}
	\label{L-summable1}
Fix $0<c<1$. For each $k\in \mathbb{Z}_{\ge 0}$ let $\phi_k$ be the power series given by
	\begin{align}
	\label{E-power series}
\phi_k(x)=(c-x)^3\sum_{m=0}^\infty \tbinom{m+k}{k}x^m.
	\end{align}
Then $ (1-c)^k|\phi_k|(c)=\mathcal{O}(k^{-3/2})$ as $k\to \infty$.
	\end{lemma}

	\begin{proof}
Let $k\in \mathbb{Z}_{\ge 0}$.
For each $m\in \mathbb{Z}_{\ge 0}$, put
	\begin{align*}
b_{k, m}&=-\tbinom{m+k}{k}+3c\tbinom{m+k+1}{k}-3c^2\tbinom{m+k+2}{k}+c^3\tbinom{m+k+3}{k}
	\\
&=\tbinom{m+k}{k}\bigl(-1+\tfrac{3c(m+k+1)}{m+1}-\tfrac{3c^2(m+k+1)(m+k+2)}{(m+1)(m+2)}+\tfrac{c^3(m+k+1)(m+k+2)(m+k+3)}{(m+1)(m+2)(m+3)}\bigr)
	\\
&=\tbinom{m+k}{k}\tfrac{g_k(m)}{(m+1)(m+2)(m+3)},
	\end{align*}
where $g_k$ is given by \eqref{E-poly}. Then
	\begin{equation}
	\label{E-expansion2}
\phi_k(x)=c^3+ x(c^3(k+1)-3c^2)+x^2\bigl(c^3\tfrac{(k+1)(k+2)}{2}-3c^2(k+1)+3c\bigr)+\sum_{m=0}^\infty x^{m+3}b_{k, m}.
	\end{equation}

Let $k_c\in \mathbb{N}$ be given by Lemma~\ref{L-root}.
Let $k\in \mathbb{N}$ with $k\ge k_c$. Then $g_k$ has the roots $t_{k, i}$ for $i=1,2, 3$ described in Lemma~\ref{L-root}. Put $m_{k, i}=\lfloor t_{k, i}\rfloor$ for $i=1, 3$, and $m_{k, 2}=\lceil t_{k, 2}\rceil$.
Then $b_{k, m}\ge 0$ exactly when $0\le m\le m_{k, 1}$ or $m_{k, 2}\le m\le m_{k, 3}$.
Increasing $k_c$ if necessary, we may assume that $c^3(k+1)-3c^2, c^3\frac{(k+1)(k+2)}{2}-3c^2(k+1)+3c>0$.
Then
	\begin{align*}
|\phi_k|(x)&=|c^3|+ x|c^3(k+1)-3c^2|+x^2\big|c^3\tfrac{(k+1)(k+2)}{2}-3c^2(k+1)+3c\big| +\sum_{m=0}^\infty x^{m+3}|b_{k, m}|
	\\
&=c^3+ x(c^3(k+1)-3c^2)+x^2\big(c^3\tfrac{(k+1)(k+2)}{2}-3c^2(k+1)+3c\big)
	\\
&\quad +\sum_{0\le m\le m_{k, 1} \textup{ \scriptsize or } m_{k, 2}\le m\le m_{k, 3}}x^{m+3}b_{k, m}-\sum_{m_{k, 1}<m<m_{k, 2} \textup{ \scriptsize or } m>m_{k, 3}} x^{m+3}b_{k, m}.
	\end{align*}
Note that $\phi_k$ converges absolutely on the open interval $(-1, 1)$. Taking $x=c$ in \eqref{E-expansion2} we get
	\begin{align*}
0=c^3+c(c^3(k+1)-3c^2)+c^2\big(c^3\tfrac{(k+1)(k+2)}{2}-3c^2(k+1)+3c\big)+\sum_{m=0}^\infty c^{m+3}b_{k, m}.
	\end{align*}
Thus
	\begin{align}
	\label{E-cancel13}
|\phi_k|(c)& =2\Big(c^3+c(c^3(k+1)-3c^2)+c^2\big(c^3\tfrac{(k+1)(k+2)}{2}-3c^2(k+1)+3c\big)
	\\
\nonumber &\quad +\sum_{0\le m\le m_{k, 1} \mbox{ \scriptsize or } m_{k, 2}\le m\le m_{k, 3}}c^{m+3}b_{k, m}\Big).
	\end{align}
Similar to \eqref{E-expansion2}, we have
	\begin{align*}
(c-x)^3&\sum_{0\le m\le m_{k, 1}}x^m\tbinom{m+k}{k}
=c^3+x(c^3(k+1)-3c^2)+x^2\big(c^3\tfrac{(k+1)(k+2)}{2}-3c^2(k+1)+3c\big)
	\\
&\qquad +\sum_{0\le m\le m_{k, 1}} x^{m+3}b_{k, m}-(c^3-3c^2x+3cx^2)x^{m_{k, 1}+1}\tbinom{m_{k, 1}+k+1}{k}
	\\
&\qquad -(c^3-3c^2x)x^{m_{k, 1}+2}\tbinom{m_{k, 1}+k+2}{k}-c^3x^{m_{k, 1}+3}\tbinom{m_{k, 1}+k+3}{k},
	\end{align*}
and
	\begin{align*}
(c-x)^3&\sum_{m_{k, 2}\le m\le m_{k, 3}}x^m\tbinom{m+k}{k}=(c^3-3c^2x+3cx^2)x^{m_{k, 2}}\tbinom{m_{k, 2}+k}{k}
	\\
&\qquad +(c^3-3c^2x)x^{m_{k, 2}+1}\tbinom{m_{k, 2}+k+1}{k} +c^3x^{m_{k, 2}+2}\tbinom{m_{k, 2}+k+2}{k}
	\\
&\qquad +\sum_{m_{k, 2}\le m\le m_{k, 3}} x^{m+3}b_{k, m} -(c^3-3c^2x+3cx^2)x^{m_{k, 3}+1}\tbinom{m_{k, 3}+k+1}{k}
	\\
&\qquad -(c^3-3c^2x)x^{m_{k, 3}+2}\tbinom{m_{k, 3}+k+2}{k}-c^3x^{m_{k, 3}+3}\tbinom{m_{k, 3}+k+3}{k}.
	\end{align*}
Taking $x=c$ in the two identities above, we get
	\begin{align*}
0&=c^3+c(c^3(k+1)-3c^2)+c^2\big(c^3\tfrac{(k+1)(k+2)}{2}-3c^2(k+1)+3c\big)
	\\
&\qquad +\sum_{0\le m\le m_{k, 1}} c^{m+3}b_{k, m}-c^3f_k(m_{k, 1}+1),
	\end{align*}
and
	\begin{align*}
0=c^3f_k(m_{k, 2})+
\sum_{m_{k, 2}\le m\le m_{k, 3}} c^{m+3}b_{k, m}-c^3f_k(m_{k, 3}+1),
	\end{align*}
where $f_k$ is defined in \eqref{E-asymp1}.
Then \eqref{E-cancel13} becomes
	\begin{align*}
|\phi_k|(c)=2c^3(f_k(m_{k, 1}+1)-f_k(m_{k, 2})+f_k(m_{k, 3}+1)),
	\end{align*}
whence
	\begin{displaymath}
(1-c)^k|\phi_k|(c)=2c^3(h_k(m_{k, 1}+1)-h_k(m_{k, 2})+h_k(m_{k, 3}+1)),
	\end{displaymath}
where $h_k$ is defined in \eqref{E-asymp}.
Note that the two sequences $\{y_{k, 4c, -}\}$ and $\{y_{k, 4c, +}\}$ are both $\frac{ck}{1-c}+\mathcal{O}(k^{1/2})$.
It follows from Lemma~\ref{L-root}.(1) that the three sequences $\{m_{k, 1}+1\}$, $\{m_{k, 2}\}$ and $\{m_{k, 3}+1\}$ are all $\frac{ck}{1-c}+\mathcal{O}(k^{1/2})$.
Thus from Lemma~\ref{L-asymp} we conclude that the three sequences $\{h_k(m_{k, 1}+1)\}, \{h_k(m_{k, 2})\}$ and $\{h_k(m_{k, 3}+1)\}$ are all $\mathcal{O}(k^{-3/2})$.
Therefore $(1-c)^k|\phi_k|(c)=\mathcal{O}(k^{-3/2})$.
	\end{proof}

We are ready to prove Lemma \ref{L-summable}.

	\begin{proof}[Proof of Lemma~\ref{L-summable}]
Let $k\in \mathbb{N}$. One has
	\begin{align*}
(1-q)^{-(k+1)}=\Bigl(\sum_{j=0}^\infty q^j\Bigr)^{k+1}=\sum_{m=0}^\infty q^m\sum_{\substack{j_1+\dots+j_{k+1}=m
	\\
j_1, \dots, j_{k+1}\ge 0}}1=\sum_{m=0}^\infty q^m\tbinom{m+k}{k},
	\end{align*}
whence
	\begin{displaymath}
(c-q)^3(1-q)^{-(k+1)}=\phi_k(q),
	\end{displaymath}
where $\phi_k$ is defined in \eqref{E-power series}. Write $\phi_k$ as $\sum_{m=0}^\infty \lambda_mx^m$ with $\lambda _m\in \mathbb{C}$.
Then
	\begin{align*}
\|r^k(c-q)^3(1-q)^{-(k+1)}\|&=\Big\|\sum_{m=0}^\infty \lambda_mr^k q^m\Big\|\le \sum_{m=0}^\infty |\lambda_m| \cdot \|r\|^k\|q\|^m
	\\
&\le \sum_{m=0}^\infty |\lambda_m| (1-c)^k c^m=(1-c)^k|\phi_k|(c).
	\end{align*}
Now the assertion follows from Lemma~\ref{L-summable1}.
	\end{proof}

Lemma \ref{L-summable} implies the following proposition.

	\begin{proposition}
	\label{P-multiplier}
Let $\mathcal{A}$ be a unital Banach algebra such that $\|ab\|\le \|a\|\cdot \|b\|$ for all $a, b\in \mathcal{A}$.
Let $0<c<1$ and $q, r\in \mathcal{A}$ such that $\|r\|\le 1-c$, $\|q\|\le c$ and $qr=rq$. Then there is some $a\in \mathcal{A}$ such that
	\begin{equation}
	\label{eq:a}
a(1-(q+r))=(1-(q+r))a=(c-q)^3.
	\end{equation}
	\end{proposition}

	\begin{proof}
Formally, the element $a$ in \eqref{eq:a} is given by
	\begin{align*}
a&=(c-q)^3(1-(q+r))^{-1} = (c-q)^3((1-q)-r)^{-1}
	\\
&=\smash[b]{(c-q)^3(1-q)^{-1}(1-r(1-q)^{-1})^{-1} =(c-q)^3(1-q)^{-1}\sum_{k=0}^\infty r^k(1-q)^{-k}}
	\\
&= \sum_{k=0}^\infty r^k(c-q)^3(1-q)^{-(k+1)}.
	\end{align*}
By Lemma~\ref{L-summable}, the last series in this expression for $a$ converges in norm, so that $a\in \mathcal{A}$ is well defined.
Since $qr=rq$, one has $qa=aq$, $ra=ar$, and $(1-q)^{-1}r=r(1-q)^{-1}$.
Thus
	\begin{align*}
a(1-(q+r))&=a(1-q)-ar
	\\
&=\sum_{k=0}^\infty r^k(c-q)^3(1-q)^{-k}-\sum_{k=0}^\infty r^{k+1}(c-q)^3(1-q)^{-(k+1)}=(c-q)^3,
	\end{align*}
and similarly $(1-(q+r))a=(c-q)^3$.
	\end{proof}

	\begin{example}
	\label{E-Heisenberg}
Let $\mathbb{H}$ be the discrete Heisenberg group with canonical generators $u_1,u_2,u_3$ defined in \eqref{H-generators}, and let $f= 4-u_1-u_1^{-1}-u_2-u_2^{-1} \in \mathbb{Z}\mathbb{H}$. Then $f= 4(1-p)$, where $p=\frac{1}{4}(u_1+u_2+u_1^{-1}+u_2^{-1})\in \mathbb{Q} \mathbb{H}$ can be viewed as a symmetric probability measure on $\mathbb{H}$. The polynomial $p^4 \in \mathbb{Q}\mathbb{H}$ can again be viewed as a probability measure on $\mathbb{H}$, and the coefficient $c\coloneq p^4_{u_3}$ of $p^4$ at $u_3$ is strictly positive. If we set $q = c\cdot u_3$ and $r = p^4 - q$, then $qr=rq$ and Proposition~\ref{P-multiplier} yields an element $a\in \ell ^1(\mathbb{H},\mathbb{R})$ such that
	\begin{displaymath}
a(1-p^4) = a(1-(q+r)) = c^3(1-u_3)^3.
	\end{displaymath}
Then $b=a(1+p+p^2+p^3)/4c^3\in \ell ^1(\mathbb{H},\mathbb{R})$ satisfies that
	\begin{equation}
	\label{eq:b}
fb = bf = bf^* = 4(1-p)b=(1-u_3)^3.
	\end{equation}
It follows that $b\in W_f$ and $\eta (b)\in \Delta ^1(X_f)$ (cf. \eqref{eq:Wf} and Proposition \ref{p:homoclinic}).

We remark in passing that the existence of such a homoclinic point $b$ was conjectured in \cite{Goll}*{page 130}; in \cite{Goll}*{Theorem 4.1.2} it was shown that there is some $b'\in \ell ^1(\mathbb{H},\mathbb{R})$ satisfying $b'f^*=(1-u_3)^9$.

\smallskip Having found a nonzero element of $\Delta ^1(X_f)$ we claim that $\Delta ^1(X_f)$ is actually dense in $X_f$. To verify this we give an ad-hoc proof based on \cite{Goell1}*{Theorem 5.1}: for every $x\in X_f$ there exists a $y\in Y_f$ with $\eta (y)=x$ (for notation we refer to \eqref{eq:eta}). Then $v\coloneqq\bar{\rho }^fy\in \{-3,\dots , 3\}^\mathbb{H}\subset \ell ^\infty (\mathbb{H},\mathbb{Z})$ (cf. \eqref{eq:Wf}) and $(\bar{\rho }^f\circ \bar{\rho }^b)(v)= \bar{\rho }^{fb}v = \bar{\rho }^{(1-u_3)^3}v \in \ell ^\infty (\mathbb{H}, \mathbb{Z})$. It follows that $\bar{\rho }^bv\in W_f$ and $(\eta \circ \bar{\rho }^b)(v)=(\eta \circ \bar{\rho }^{bf})(y) = \rho ^{bf}(\eta (y)) =\rho ^{(1-u_3)^3}x$ by \eqref{eq:b}. We set $\mathcal{V} = \{-3,\dots , 3\}^\mathbb{H}\subset \ell ^\infty (\mathbb{H},\mathbb{Z})$ and conclude that $(\eta \circ \bar{\rho }^b)(\mathcal{V})\supseteq \rho ^{(1-u_3)^3}(X_f)$.

We recall that $X_f = \widehat{\mathbb{Z}\mathbb{H}/(f)} = (f)^\perp \subset \widehat{\mathbb{Z}\mathbb{H}}$ (cf. \eqref{eq:Xf}). If $\rho ^{1-u_3}(X_f)\subsetneq X_f$ there exists an element $h\in \mathbb{Z}\mathbb{H}$ such that $h\notin (f)$ and $\langle h,\rho ^{1-u_3}x \rangle = \langle h(1-u_3),x \rangle =1$ for every $x\in X_f$. Hence $h(1-u_3) = (1-u_3)h\in (f)$, i.e. $(1-u_3)h= gf$ for some $g\in \mathbb{Z}\mathbb{H}$.

We denote by $\langle u_3 \rangle $ the subgroup of $\mathbb{H}$ generated by $u_3$, set $\mathbb{H}' = \mathbb{H}/\langle u_3\rangle \cong \mathbb{Z}^2$, and denote by $\pi \colon \mathbb{Z}\mathbb{H} \to \mathbb{Z}\mathbb{H}' \cong \mathbb{Z}\mathbb{H}/(z_3-1)\mathbb{Z}\mathbb{H}$ the group ring homomorphism corresponding to the quotient map $\mathbb{H}\to \mathbb{H}'$. As $f$ is not divisible by $1-u_3$, $\pi (f)\ne 0$, but $\pi (gf)=\pi (g)\pi (f)=0$. Since $\mathbb{Z}\mathbb{H}' \cong \mathbb{Z}\mathbb{Z}^2$ is an integral domain we obtain that $\pi (g)=0$, i.e. that $g = (1-u_3)g'$ for some $g'\in \mathbb{Z}\mathbb{H}$. Since $\mathbb{Z}\mathbb{H}$ has no nontrivial zero divisors (cf. e.g. \cite{Passman}*{Theorem 13.1.11}) we obtain that $h=g'f$, contrary to our hypothesis that $h\notin \mathbb{Z}\mathbb{H}f$.

This contradiction implies that $X_f = \rho ^{1-u_3}(X_f) = \rho ^{(1-u_3)^3}(X_f) = (\eta \circ \bar{\rho }^b)(\mathcal{V})$. Since $\eta \circ \bar{\rho }^b\colon \mathcal{V} \to X_f$ is continuous and $\mathbb{Z}\mathbb{H}\cap \mathcal{V}$ is dense in $\mathcal{V}$ (both in the product topology on $\mathcal{V}$) we conclude that $(\eta \circ \bar{\rho }^b)(\mathbb{Z}\mathbb{H})$ is dense in $X_f$. Finally we note that $(\eta \circ \bar{\rho }^b)(\mathbb{Z}\mathbb{H})\subset \Delta ^1(X_f)$, so that $\Delta ^1(X_f)$ is dense in $X_f$, as promised.

According to Proposition \ref{p:homoclinic} this shows that $\lambda _f$ is intrinsically ergodic.
	\end{example}

The following corollary of Proposition \ref{P-multiplier} will allow us to extend the argument in Example \ref{E-Heisenberg} to the much more general setting of Theorem \ref{T-Harmonic}.

	\begin{corollary}
	\label{C-summable}
Assume that $\Gamma$ is infinite and not virtually $\mathbb{Z}$ or $\mathbb{Z}^2$. Let $p$ be a finitely supported symmetric probability measure on $\Gamma $ such that $\textup{supp}(p)$ generates $\Gamma $. By a result of Varopoulos {\textup(}cf. \cite{Varopoulos}, \cite{Furman}*{Theorem 2.1}, \cite{Woess}*{Theorem 3.24}{\textup)}, $\sum_{j=0}^\infty p^j$ converges in $\|\cdot \|_\infty$ to some $\omega$ in $C_0(\Gamma ,\mathbb{R})$. Then $(1-s)^3\omega\in \ell ^1(\Gamma ,\mathbb{R})$ for every $s$ in the center of $\Gamma$.
	\end{corollary}

	\begin{proof}
It is easily checked that
	\begin{displaymath}
\omega (1-p)=1.
	\end{displaymath}
Let $s\ne 1_\Gamma $ be a central element of $\Gamma $. Then $s\in \textup{supp}(p^k)\smallsetminus \{1_\Gamma \}$ for some $k\in \mathbb{N}$. As in Example \ref{E-Heisenberg} there is some $a$ in $\ell ^1(\Gamma ,\mathbb{R})$ such that
	\begin{displaymath}
a(1-p^k)=(1-s)^3.
	\end{displaymath}
Then $b=a\sum_{j=0}^{k-1}p^j$ is in $\ell ^1(\Gamma ,\mathbb{R})$ and
	\begin{displaymath}
b(1-p)=(1-s)^3.
	\end{displaymath}
Note that $(1-s)^3\omega\in C_0(\Gamma ,\mathbb{R})$ and
	\begin{displaymath}
(1-s)^3\omega (1-p)=(1-s)^3=b(1-p).
	\end{displaymath}
It is well known that if $x\in C_0(\Gamma ,\mathbb{R})$ satisfies $x(1-p)=0$, then $x=0$. Thus
	\begin{displaymath}
(1-s)^3 \omega=b\in \ell ^1(\Gamma ,\mathbb{R}).
\tag*{\qedsymbol}
	\end{displaymath}
\renewcommand{\qedsymbol}{}\vspace{-\baselineskip}
	\end{proof}

	\begin{proof}[Proof of Theorem \ref{T-Harmonic}]
Since $f$ is well-balanced, one has $f=f_{1_\Gamma}(1-p)$ for some symmetric probability measure $p$ on $\Gamma$ such that $\textup{supp}(p)$ generates $\Gamma$.

Since $\Gamma$ is not virtually $\mathbb{Z}$ or $\mathbb{Z}^2$, $\omega=\sum_{j=0}^\infty p^j$ is in $C_0(\Gamma ,\mathbb{R})$. Then
	\begin{displaymath}
(f_{1_\Gamma}^{-1}\omega) f=1.
	\end{displaymath}
By \cite{Bowen-Li}*{Theorem 4.1 and Lemma 4.10}, the group $\Delta(X_f)$ of homoclinic points of $\lambda _f$ is dense in $X_f$ and is the $\Gamma $-invariant subgroup of $X_f$ generated by $\eta (f_{1_\Gamma}^{-1}\omega)$. Thus it suffices to show that $\eta (f_{1_\Gamma}^{-1}\omega)$ is in the closure of $\Delta ^1(X_f)$.

By assumption we can find a sequence $(s_n)_{n\ge1}$ in the center $H$ of $\Gamma $ such that for any finite subset $F$ of $\Gamma$ one has $s_n, s_n^2, s_n^3\not\in F$ for all large enough $n$. Then $\eta ((1-s_n)^3f_{1_\Gamma}^{-1}\omega)$ converges to $\eta (f_{1_\Gamma}^{-1}\omega)$ as $n\to \infty$. By Corollary~\ref{C-summable} one has $(1-s_n)^3f_{1_\Gamma}^{-1}\omega\in \ell ^1(\Gamma ,\mathbb{R})$ and hence $\eta ((1-s_n)^3f_{1_\Gamma}^{-1}\omega)\in \Delta ^1(X_f)$ for each $n$. Therefore $\eta (f_{1_\Gamma}^{-1}\omega)$ lies in the closure of $\Delta ^1(X_f)$, as desired.
	\end{proof}

	\begin{remarks}
	\label{R-Z2}
(1) For $\Gamma =\mathbb{Z}^d$ with $d\ge1$, \cite{LSV2}*{Corollary 3.4} shows that any atoral polynomial $f\in \mathbb{Z}[\mathbb{Z}^d]$ which is not a unit in $\mathbb{Z}[\mathbb{Z}^d]$ satisfies that $\Delta ^1(X_f)$ is dense in $X_f$ (for the definition of atorality we refer to \cite{LSV2}*{Definition 2.1 and Proposition 2.2}). In particular, Theorem~\ref{T-Harmonic} also holds for $\Gamma=\mathbb{Z}^2$. Does Theorem~\ref{T-Harmonic} hold for virtually $\mathbb{Z}^2$?

\medskip (2) For the polynomial $h = 2-u_1-u_2 \in \mathbb{Z}\mathbb{H}$ there exists a nonzero element $w\in \ell ^1(\mathbb{H},\mathbb{R})$ such that $wh = hw = (1-u_3)^2$ (cf. \cite{Goell1}*{Theorem 4.2}). As in Subsection \ref{sss:expansive}, the map $\bar{\rho }^w\colon \ell ^\infty (\mathbb{H},\mathbb{Z}) \to \ell ^\infty (\mathbb{H},\mathbb{R})$ is continuous in the weak$^*$-topology on closed, bounded subsets of $\ell ^\infty (\mathbb{H},\mathbb{Z})$, and the map $\xi \coloneqq \eta \circ \bar{\rho }^w \colon \ell ^\infty (\mathbb{H},\mathbb{Z})\to \mathbb{T}^\mathbb{H}$ satisfies that $\xi (\{-1,0,1\}^\mathbb{H}) = X_{h}$ (cf. \cite{Goell1}*{Theorem 5.1}). It follows that $\xi (\mathbb{Z}\mathbb{H})\subset \Delta ^1(X_{h})$. Hence $\Delta ^1(X_h)$ is dense in $X_h$ and $\lambda _h$ is intrinsically ergodic.

\smallskip Is there any way to use an argument similar to Proposition~\ref{P-multiplier} to prove the existence of summable homoclinic points for this and other `asymmetric' elements of $h \in \mathbb{Z}\mathbb{H}$ ?
	\end{remarks}

	\end{document}